\documentclass[11pt, a4paper]{amsart}

\usepackage{amsfonts,amsmath,amssymb, amscd}

\newtheorem{theorem}{Theorem}[section]
\newtheorem{lemma}[theorem]{Lemma}
\newtheorem{definition}[theorem]{Definition}
\newtheorem{prop}[theorem]{Proposition}
\newtheorem{corollary}[theorem]{Corollary}

\theoremstyle{definition}
\newtheorem{rem}[theorem]{Remark}

\newtheorem{example}[theorem]{Example}

\newcommand\pf{\begin{proof}}
\newcommand\epf{\end{proof}}

\newcommand{\cmddotrtimes}{\mathop{\raisebox{0.2ex}{\makebox[0.86em][l]{${\scriptstyle>\mathrel{\mkern-4mu}\lessdot}$}}\raisebox{0.12ex}{$ \shortmid$}}}

\numberwithin{equation}{section}

\title[Quotient sheaves of algebraic supergroups]
{Quotient sheaves of algebraic supergroups are superschemes}

\author[A.\ Masuoka]{Akira Masuoka}
\address{Akira Masuoka: 
Institute of Mathematics, 
University of Tsukuba, 
Ibaraki 305-8571, Japan}
\email{akira@math.tsukuba.ac.jp}

\author[A.\ N.\ Zubkov]{Alexander N.~Zubkov}
\address{Alexander N.~Zubkov: 
Department of Mathematics, 
Omsk State Pedagogical University, 
Omsk--644099, Russia}
\email{a.zubkov@yahoo.com}

\begin{document}


\noindent
{\sc}

\medskip
\noindent
{\sc}

\hspace{3cm}

\maketitle

\begin{abstract}
To generalize some fundamental results on group schemes to the super context, we study the quotient sheaf 
$G \tilde{/} H$ of an algebraic supergroup $G$ by its closed supersubgroup $H$,
in arbitrary characteristic $\neq$ 2. Our main theorem states that $G \tilde{/} H$
is a Noetherian superscheme. This together with derived results give positive answers to interesting
questions posed by J. Brundan.
\end{abstract}

\section*{Introduction}
This paper is concerned with generalizing theory of algebraic groups, as presented by 
Demazure and Gabriel \cite{dg} or Jantzen \cite{jan},
to the super context from a \emph{functorial} view-point. 
Recent papers with the same concern include \cite{br}, \cite{z} and \cite{z1}.
Unless otherwise stated, we work over a fixed field $K$ whose characteristic is different from 2.
We study the quotient sheaf $G \tilde{/} H$ of an algebraic supergroup $G$ by its closed supersubgroup $H$,
strongly motivated by those interesting questions posed by J. Brundan which will be noted below.

Let us recall from \cite{z} some basic definitions and known results.
Those vector spaces (over the field $K$ as above) which are graded by $\mathbb{Z}_2 = \{0, 1\}$ 
form a tensor category, $\mathsf{SMod}_K$,
with the canonical symmetry. Objects defined in this symmetric tensor category are called
with the adjective `super' attached. For example, an algebra object in $\mathsf{SMod}_K$ 
is called a \emph{superalgebra}. All superalgebras including Hopf superalgebras are
assumed to be supercommutative. A \emph{$K$-functor} (resp., a \emph{supergroup}) is a set-valued 
(resp., group-valued) functor defined on the category $\mathsf{SAlg}_K$ of superalgebras.
The $K$-functors includes the following subclasses:
\begin{equation*}
\mbox{(affine superschemes)} \subset \mbox{(superschemes)} \subset \mbox{($K$-sheaves)}. 
\end{equation*}
Every $K$-functor $X$ has uniquely a 
$K$-sheaf $\widetilde{X}$ (with respect to the fppf topology) 
together with a morphism $X \to \widetilde{X}$ of $K$-functors which have the 
obvious universal property; this $\widetilde{X}$ is called the \emph{sheafification} of $X$. 

By an \emph{algebraic supergroup} we always mean an algebraic affine supergroup, or namely 
a supergroup $G$ which is represented by a finitely generated Hopf superalgebra $K[G]$.
A \emph{closed supersubgroup} of $G$ is a supergroup $H$ represented by a quotient 
Hopf superalgebra of $K[G]$.  Let $G$, $H$ be as just defined.
The $K$-functor which associates to every superalgebra $R$, the set $G(R)/H(R)$ of right cosets 
is called the \emph{naive quotient}, denoted $(G/H)_{(n)}$. The sheafification 
of $(G/H)_{(n)}$ is denoted by $G \tilde{/} H$. It is proved in \cite{z} (see also \cite{amas}) that 
$G \tilde{/} H$ is an affine supergroup, if $H$ is normal in $G$. It is important and is our concern to study
$G \tilde{/} H$ when $H$ is not necessarily normal. Note that if the ranges of $G$, $H$ are restricted to 
the category $\mathsf{Alg}_K$ of (purely even) algebras, we have an algebraic 
(non-super) group denoted $G_{res}$, and its closed subgroup 
denoted $H_{res}$. We are also concerned with the relation between $G \tilde{/} H$ and 
the quotient $G_{res} \tilde{/} H_{res}$ in the classical, non-super situation.

The questions on $G \tilde{/} H$ posed by Brundan, which were brought to the second named author by a private
communication, are the following.

\begin{itemize}
\item[(Q1)] 
Is $G \tilde{/} H$ necessarily a superscheme?
\item[(Q2)]
Is $G \tilde{/} H$ affine whenever the algebraic group $H_{res}$ is geometrically reductive? 
\end{itemize}

On the other hand Brundan \cite{br} defined a kind of quotients of $G$ by $H$, which we call the 
\emph{Brundan quotient}, as a superscheme with some desired properties; see below.
The Brundan quotient looks different from the explicitly constructed $G \tilde{/} H$.
Therefore we have the following in mind.

\begin{itemize}
\item[(Q3)] 
Does the Brundan quotient always exist, and coincide with $G \tilde{/} H$?
\end{itemize}

This paper answers these questions all in the positive. First of all, our main theorem, which 
answers (Q1) positively, is the following: 

\begin{theorem}\label{mainthm}
Let $G$ be an algebraic supergroup, and let $H$ be a closed supersubgroup of $G$. Then  
the $K$-sheaf $G \tilde{/} H$ is a Noetherian superscheme.
\end{theorem}

The same statement holds true for the $K$-sheaf $G \tilde{\backslash} H$ which is defined to be the sheafification
of the naive quotient $(G \backslash H)_{(n)}$ of left cosets. This and other results on $G \tilde{\backslash} H$
follow from the corresponding results on $G \tilde{/} H$ applied to the opposite supergroups 
$G^{op} \supseteq H^{op}$. 
As a corollary to the proof of the theorem above, we have that $G \tilde{/} H$ is affine iff $G_{res} \tilde{/} H_{res}$ is affine;
see Corollary \ref{cor1}. This answers (Q2) in the positive since it follows from 
Cline et al. \cite{cps}, Corollary 4.5,  that 
$G_{res} \tilde{/} H_{res}$ is affine if $H_{res}$ is geometrically reductive; see Remark \ref{folklore} for more details.

The \emph{Brundan quotient} of $G$ by $H$, cited above, is a pair $(X, \pi)$ of a Noetherian superscheme $X$ and
a morphism $\pi : G \to X$ such that 
\begin{itemize}
\item[(1)]
$\pi$ is affine and faithfully flat\ (see Section 9 for definitions), 
\item[(2)]
$\pi$ factors (necessarily in a unique way) through the quotient morphism $G \to (G/H)_{(n)}$, and
\item[(3)]
if a morphism $G \to Y$ to a superscheme $Y$ factors through $G \to (G/H)_{(n)}$, it uniquely factors through $\pi$.
\end{itemize}
\noindent
We will prove that the quotient morphism $G \to G \tilde{/} H$ has the property (1); see Corollary \ref{affandfaith}. 
Since it has obviously the properties (2), (3), the main theorem above answers (Q3) in the positive.

Recall that Theorem \ref{mainthm} above was proved in the classical, non-super situation by Demazure and
Gabriel \cite{dg}, III, \S 3, 5.4; thus we know already that with our notation, $G_{res} \tilde{/} H_{res}$ is a
Noetherian scheme. Our proof of the theorem reduces to this classical result, investigating
the relation of  $G \tilde{/} H$ with $G_{res} \tilde{/} H_{res}$. 
Our method of the proof is a combination of geometric and Hopf-algebraic ones, which work effectively 
for global and local questions, respectively. 
Our geometric method is represented by Comparison Theorem \ref{comparison}, which generalizes in the super context 
\emph{th$\Acute{e}$or$\Grave{e}$me de comparaison} by Demazure and Gabriel \cite{dg}, I, \S 1, 4.4. 
Besides the functorial approach, there is another approach to supergeometry through \emph{geometric superspaces},
which are topological spaces with structure sheaves of superalgebras. Roughly speaking, the Comparison Theorem tells us
that the two approaches are equivalent at the level of superschemes; it enables us to obtain useful topological 
information on superschemes.
We emphasize that our proof of the theorem is not
merely a translation of the proof by Demazure and Gabriel, giving more detailed explanations.
As our Hopf-algebraic method the tensor product decomposition of a Hopf superalgebra plays an
important role; see Proposition \ref{prop-isom} which is reproduced from \cite{amas}. 
This result tells us that $G$ is moderately related with $G_{res}$,
and it enables us at some crucial steps to reduce our argument to the non-super context.

The main body of this paper consists of Sections 8 and 9. Section 8 is devoted mostly to proving the main theorem above,
while Section 9 shows some further properties of $G \tilde{/} H$; the latter contains, besides the corollary last referred to, 
Proposition \ref{ev of factor is factor of ev} which states especially that $G \tilde{/} H$, with its range restricted to 
$\mathsf{Alg}_K$, turns to coincide with $G_{res} \tilde{/} H_{res}$. The preceding seven sections and the last, rather
independent Section 10 are devoted to preliminaries for the two main sections. Let us describe briefly the contents of these
eight sections. 
Section 1 discusses direct limits, which are used to construct geometric superspaces. Section 2 gives
basic results on super(co)algebras and their super(co)modules. Section 3 summarizes basics on $K$-functors and sheaves.
In Section 4, we discuss geometric superspaces, and construct such a space from a $K$-functor. In Section 5, we 
formulate the Comparison Theorem cited above, and prove it. Section 6 discusses the supergrassmannian, which turns out to be
a model of quotient sheaves, and into which every quotient sheaf $G \tilde{/} H$ can be embedded.  In Section 7, we discuss
the quotient \emph{dur sheaf} $X \tilde{\tilde{/}} G$ associated to an affine superscheme
$X$ on which an affine supergroup $G$ acts; with $G$, $H$ as before, the discussion will be applied to 
$U \tilde{/} H$, where $U$ is such an affine open supersubscheme of $G$ that is stable under the right multiplication by $H$.
Theorem \ref{superOberst} gives some necessary and sufficient conditions for us to have that
the $G$-action on $X$ is free, and  $X \tilde{\tilde{/}} G$ is affine. A part of the proof of the theorem will be postponed
until Section 10. The postponed proof uses the bozonization technique, which is applied to a more general situation 
(i.e., to the braided tensor category of Yetter-Drinfeld modules which generalizes $\mathsf{SMod}_K$) than is needed for the sake
of its own interest.

\section{Direct limits}\label{directlimit}

Let $\mathcal{A}$ be a category. For an object $A\in Ob \ \mathcal{A}$, denote the functors $B \mapsto \mathrm{Mor}_{\mathcal{A}}(B,
A)$ and $B \mapsto \mathrm{Mor}_{\mathcal{A}}(A, B),\ B \in Ob \ {\mathcal {A}},$ by $h_A$ and $h^A$, respectively. The functors $h_A$ and $h^A$ are called
{\it representable, contravariant} and {\it covariant} functors, respectively. Denote the category of covariant functors from $\mathcal{A}$ to $\mathsf{Sets}$ by $\mathsf{Sets}^{\mathcal{A}}$. 

\begin{lemma}\label{yoneda}
(Yoneda's Lemma, cf. \cite{mc}, p.61)\
For any covariant functor $f: \mathcal{A} \to \mathsf{Sets}$ there is an bijection
$$\mathrm{Mor}_{\mathsf{Sets}^{\mathcal{A}}}(h^A, f) \overset{\simeq}{\longrightarrow} f(A),\ A \in Ob \ \mathcal{A},$$
functorial in both arguments. 
\end{lemma}

\begin{proof}
The bijection is defined by ${\bf g}\mapsto x_{\bf g}={\bf g}(A)(\mathrm{id}_A)$ and its inverse by $x \mapsto {\bf g}_x$, where
${\bf g}_x(\alpha)=f(\alpha)(x),\ {\bf g}\in \mathrm{Mor}_{\mathsf{Sets}^{\mathcal{A}}}(h^A, f),\ x\in f(A),\ \alpha\in h^A(C),\ C\in Ob \ {\mathcal{A}}$.
\end{proof}

For the contravariant version of the above lemma see \cite{bd}, Theorem 1.6.

Set $f=h^B$. By Yoneda's Lemma each $\phi\in \mathrm{Mor}_{\mathcal{A}}(B, A)$ defines a morphism of functors $h^A \to h^B$ that is denoted by
${\bf h}(\phi)$. 

Recall the definition of a direct limit.
Let $f : \mathcal{A} \to \mathcal{B}$ be a covariant functor. A {\it direct limit} of $f$,
denoted by $\lim\limits_{\rightarrow} f$, is an object $Z\in Ob \ \mathcal{B}$ and a collection
of morphisms $\{i^f_X : f(X)\to Z\}_{X\in Ob \ \mathcal{A}}$ such that:
\begin{itemize}
\item [1.]
For any two objects $X, Y\in  Ob \ \mathcal{A}$ and an arbitrary 
morphism $\alpha\in \mathrm{Mor}_{\mathcal{A}}(X, Y)$ the diagram
$$\begin{array}{ccc}
 & Z & \\
i^f_X \ \nearrow & & \nwarrow \ i^f_Y  \\
f(X) & \stackrel{f(\alpha)}{\rightarrow} & f(Y)
\end{array}
$$
is commutative.
\item[2.]
If an object $Z'$ and a collection of morphisms $\{j^f_X : f(X)\to Z'\}_{X\in Ob \ \mathcal{A}}$
satisfy the above condition, then there is a unique morphism $g : Z\to Z'$ such that
$j^f_X=g i^f_X, X\in Ob \ \mathcal{A}$.
\end{itemize}

We omit the upper index $f$ if it does not lead to confusion.
  
If $\lim\limits_{\rightarrow} f$ exists, then it is unique up to an isomorphism.
One can define symmetrically a {\it projective limit} $\lim\limits_{\leftarrow} f$ of a functor $f$. It is also unique up to an isomorphism, whenever it exists (cf. \cite{bd}, Corollary 3.2 and remarks below). 

Finally, if $f : \mathcal{A} \to \mathcal{B}$ and $g : \mathcal{A} \to \mathcal{B}$ are covariant functors, then any morphism of functors ${\bf h} : f\to g$ induces a morphism $\lim\limits_{\rightarrow} {\bf h} : \lim\limits_{\rightarrow} f\to
\lim\limits_{\rightarrow} g$, provided the direct limits exist. To construct it just consider the collection 
$\{ i^g_A {\bf h}(A) : f(A)\to \lim\limits_{\rightarrow} g\}_{A\in Ob \ \mathcal{A}}$.

\begin{example}\label{limitset}
Let $f : \mathcal{A}\to \mathsf{Sets}$ be a covariant functor. By \cite{bd}, Proposition 3.4, $\lim\limits_{\rightarrow} f$ exists and equals 
a quotient set of $\bigsqcup_{A\in Ob \ \mathcal{A}} f(A)$ by the smallest equivalence relation
that contains all pairs $(a, b)$ such that $f(\alpha)(a)=b, a\in f(A),\ b \in f(B)$ and $\alpha \in \mathrm{Mor}_{\mathcal{A}}(A, B)$.
\end{example}

Let $\alpha \in \mathrm{Mor}_{\mathcal{A}}(A, B)$ and $\beta\in \mathrm{Mor}_{\mathcal{A}}(A, C)$. An object
$D \in Ob \ \mathcal{A}$ with two morphisms $\gamma\in \mathrm{Mor}_{\mathcal{A}}(B, D)$ and $\delta\in \mathrm{Mor}_{\mathcal{A}}(C, D)$ such that 
$\delta\beta=\gamma\alpha$ is called a {\it compositum} of $\alpha$ and $\beta$.
\begin{example}\label{field}
Let $\mathsf{F}_K$ be a category of field extensions $K\subseteq F$ whose morphisms are $K$-algebra morphisms.
For any two morphisms $\alpha : F\to L_1$ and $\beta : F\to L_2$ in $\mathsf{F}_K$ for a compositum of $\alpha$ and $\beta$
one can choose a compositum of fields $L_1 L_2$ (over $F$) with the canonical inclusions $L_1\to L_1 L_2$ and $L_2\to L_1 L_2$.
\end{example}

The proof of the following lemma is an elementary exercise that is left to the reader.

\begin{lemma}\label{limit}
Let $f : \mathcal{A}\to \mathsf{Sets}$ be a covariant functor. If arbitrary two morphisms in $\mathcal{A}$ admit a compositum, then: 

1. Any two elements $a\in f(A), b\in f(B)$ 
are equivalent iff there is an object $C \in Ob \ \mathcal{A}$ and morphisms $\alpha\in \mathrm{Mor}_{\mathcal{A}}(A, C), \beta \in \mathrm{Mor}_{\mathcal{A}}(B, C)$
such that $f(\alpha)(a)=f(\beta)(b)$; 

2. If $h$ is a subfunctor of $f$, then $\lim\limits_{\rightarrow} h$ is a subset of
$\lim\limits_{\rightarrow} f$;

3. If $h_1$ and $h_2$ are subfunctors of $f$, then $\lim\limits_{\rightarrow} (h_1\bigcap h_2)=
\lim\limits_{\rightarrow} h_1\bigcap\lim\limits_{\rightarrow} h_2$.  
\end{lemma}

If $x\in \lim\limits_{\rightarrow} f$, then a subfunctor 
$f_x$ of $f$ is defined by $f_x(A)=i_A^{-1}(x)$; this $f_x(A)$ can be identified with
$x\bigcap f(A)$, whenever $x$ is identified with an equivalence class. Thus $f=\bigsqcup_{x\in
\lim\limits_{\rightarrow} f} f_x$, and every $f_x$ is an indecomposable functor. 

Let $f : \mathcal{A}\to \mathsf{Sets}$ be as above. Consider the category $\mathcal{M}_f$ whose objects are 
pairs $(A, x),\ A\in Ob \ \mathcal{A},\ x \in f(A)$ and morphisms $(A, x) \to (B, y)$ are
morphisms $\phi \in \mathrm{Mor}_{\mathcal{A}}(A, B)$ such that 
$f(\phi)(x)=y$. We have the functor $\delta_f : (\mathcal{M}_f)^{\circ} \to
\mathsf{Sets}^{\mathcal{A}}$ defined by $\delta_f((A, x))=h^A,\ \delta_f(\phi)={\bf h}(\phi)$.
Here, $(\mathcal{M}_f)^{\circ}$ denotes the opposite category of $\mathcal{M}_f$. 

\begin{lemma}\label{limdelta}
It holds that $\lim\limits_{\rightarrow}\delta_f=f$.
\end{lemma}

\begin{proof}
By Yoneda's Lemma 
morphisms $i_{(A, x)}={\bf g}_x : h^A\to f$ satisfy the first condition of the definition of a direct limit. For a collection
of morphisms $i'_{(A, x)} : h^A\to h$ as in the second condition the morphism ${\bf g} : f\to h$ is (uniquely) defined by 
${\bf g}(A)(x)=i'_{(A, x)}(\mathrm{id}_A),\ A \in Ob \ \mathcal{A},\ x \in f(A)$. 
\end{proof}
  
\begin{lemma}\label{prlimit}
(cf. \cite{bd}, Proposition 3.7)\
If a covariant functor $f : \mathcal{A} \to \mathcal{B}$ has a direct limit,
then for any $B\in Ob\ \mathcal{B}$ we have the natural isomorphism $\mathrm{Mor}_{\mathcal{B}}(\lim\limits_{\rightarrow} f, B) \simeq
\lim\limits_{\leftarrow} h_B \circ f$.
\end{lemma}
Let $f : {\mathcal A}\times {\mathcal B}\to \mathsf{Sets}$ be a bifunctor. We have a functor $g : A\to f(A, ?)$ from $\mathcal A$ to $\mathsf{Sets}^{\mathcal B}$. Since $\lim\limits_{\stackrel{\rightarrow}{A}} f(A, B)$ exists for any $B\in {\mathcal B}$, we have also a 
functor $h : B\to \lim\limits_{\stackrel{\rightarrow}{A}} f(A, B)$.  
\begin{lemma}\label{restriction}
Assume that $g=\lim\limits_{\stackrel{\rightarrow}{A}} f(A, ?)$ exists. Then $g\simeq h$.
\end{lemma}
\begin{proof}
Since $\{i^{f(A, ?)}_A(B)\}_{A\in Ob \ {\mathcal A}}$
satisfies the first condition for $\lim\limits_{\stackrel{\rightarrow}{A}} f(A, B)$, it defines ${\bf u}_B : \lim\limits_{\stackrel{\rightarrow}{A}} f(A, B)\to g(B)$ that is functorial in $B$. Therefore, we have a morphism ${\bf u} : h\to g$. Symmetrically, any collection $\{i^{f(?, B)}_A\}_{B\in Ob \ {\mathcal B}}$ defines a morphism $j_A : f(A, ?)\to h$. Moreover, $j_{A'}f(\alpha, ?)=j_A$ for all $A, A'\in Ob \ {\mathcal A}, \alpha\in
\mathrm{Mor}_{\mathcal A}(A, A')$. By the universality, there is a morphism ${\bf v} : g\to h$ and ${\bf u}{\bf v}= \mathrm{id}_g, {\bf v}{\bf u}= \mathrm{id}_h$. 
\end{proof}
The following lemma is obvious; see Example \ref{limitset}.
\begin{lemma}\label{fiberprodset}
Let $f\in \mathsf{Sets}^{\mathcal A}$ and $C\to D, \lim\limits_{\rightarrow} f\to D$ be maps of sets. Then
$C\times_D \lim\limits_{\rightarrow} f\simeq \lim\limits_{\rightarrow} C\times_D f$, where $(C\times_D f)(A)=C\times_D f(A),
A\in Ob \ {\mathcal A}$. 
\end{lemma}
\begin{lemma}\label{equivofcategories}
Let $f : \mathcal{A}\to\mathcal{B}$ and $g : \mathcal{B}\to\mathcal{A}$ be two (covariant) functors. If 
$f$ is a left adjoint to $g$ and both $f$ and $g$ are full and faithful, then $f$ and $g$ are equivalences 
which are quasi-inverses of each other.
\end{lemma}
\begin{proof}
By Proposition 1.13', \cite{bd}, for any $B\in \mathcal{B}$ there is a natural isomorphism $B\simeq fg(B)$. Proposition 1.19,
\cite{bd}, infers that $f$ is an equivalence. Let $l :  \mathcal{B}\to\mathcal{A}$ be its quasi-inverse. Proposition 1.16 and Corollary 1.11, \cite{bd}, imply that $l\simeq g$. 
\end{proof}

\section{Supermodules and supercomodules}\label{supermodule}

A \emph{supervector space} is a vector space graded by the group $\mathbb{Z}_2 = \{0, 1\}$. Given such a
vector space $V$, the homogeneous components are denoted by $V_0$, $V_1$. The degree of a homogeneous 
element, say $v$, is denoted by $|v|$. Let $\mathsf{SMod}_K$ denote the $K$-linear abelian category of supervector spaces.
This forms naturally a tensor category with the canonical symmetry
$$V \otimes W \overset{\simeq}{\longrightarrow} W \otimes V, \quad v \otimes w \mapsto (-1)^{|v||w|} w \otimes v,$$
where $V, W \in \mathsf{SMod}_K$.
Objects defined in this symmetric tensor category are called with the adjective `super' attached. For example,
a \emph{(Hopf) superalgebra} is a (Hopf) algebra object in $\mathsf{SMod}_K$. Superalgebras in any kind,
including Hopf superalgebras, are all assumed to be supercommutative
so that $ab = (-1)^{|a||b|}ba$, unless otherwise stated. 
Let $\mathsf{SAlg}_K$ denote the category of (supercommutative) superalgebras.

Given $A \in \mathsf{SAlg}_K$, we let ${}_A\mathsf{SMod}$,\ $\mathsf{SMod}_A$ denote the category of left and respectively, 
right $A$-supermodules; an object in ${}_A\mathsf{SMod}$, for example, is precisely a left $A$-module object in $\mathsf{SMod}_K$.
The two categories just defined are identified if we regard each $M \in \mathsf{SMod}_A$ as an object in ${}_A\mathsf{SMod}$ 
by defining the left $A$-action  
\begin{equation}\label{leftaction}
am := (-1)^{|a||m|}ma, \quad a \in A,\ m \in M
\end{equation}
on the supervector space $M$. We remark that $M$ thus turns into an $(A, A)$-superbimodule.
\begin{prop}\label{Noetherian}
For $A \in \mathsf{SAlg}_K$, the following are equivalent:
\begin{itemize}
\item[(1)]
$A$ is left Noetherian as a ring;
\item[(2)]
$A$ is right Noetherian as a ring;
\item[(3)]
The superideals in $A$ satisfy the ascending chain condition.
\end{itemize}
\end{prop}

If these conditions are satisfied we say that $A$ is \emph{Noetherian}.

\begin{proof}
Obviously, (1) $\Rightarrow$ (3). To prove the converse, assume (3). 
In ${}_A\mathsf{SMod}$, construct the direct sum $A \oplus A[1]$ of $A$ and its degree shift $A[1]$.
Then this direct sum is Noetherian in ${}_A\mathsf{SMod}$. On the other hand, 
we can make $A \otimes \mathbb{Z}_2$ into an object in ${}_A\mathsf{SMod}$ by defining 
$$ |b \otimes i| = i, \quad a(b \otimes i) = ab \otimes (|a| + i)$$
for $a \in A$, $b \otimes i \in A \otimes \mathbb{Z}_2$.
We see that $(a, b) \mapsto a \otimes |a| + b \otimes |b| $ gives an isomorphism 
$A \oplus A[1] \overset{\simeq}{\longrightarrow} A \otimes \mathbb{Z}_2$ in ${}_A\mathsf{SMod}$, 
which implies that $A \otimes \mathbb{Z}_2$ is Noetherian. 
Therefore we must have (1). Similarly we see (2) $\Leftrightarrow$ (3). 
\end{proof} 

Recall from \cite{amas}, Lemma 5.1(1) or \cite{z}, p. 721, the following result.

\begin{lemma}\label{faithfulflat}
Let $A \in \mathsf{SAlg}_K$ and $M \in \mathsf{SMod}_A$. The following are equivalent: 
\begin{itemize}
\item[(1)]
$M$ is faithfully flat as a left $A$-module;
\item[(2)]
$M$ is faithfully flat as a right $A$-module;
\item[(3)]
The functor $M \otimes_A : {}_A\mathsf{SMod} \to \mathsf{SMod}_K$ is faithfully exact.
\end{itemize}
\end{lemma}

If these conditions are satisfied we say that $M$ is \emph{faithfully flat over} $A$, or that 
$M$ is a \emph{faithfully flat $A$-module}. Recall also that the equivalence above remains to hold
if we remove $``$faithfully" from all the conditions.

Next, 
let $C$ be a supercoalgebra.  Let $\mathsf{SMod}^C$, ${}^C\mathsf{SMod}$ 
denote the categories of right and respectively, left $C$-supercomodules
If $C$ is regarded as an ordinary coalgebra, we let
$\mathsf{Mod}^C$, ${}^C\mathsf{Mod}$ denote the categories of right and respectively, left $C$-comodules.

\begin{prop}\label{injcomod}
Let $C$ be as above. For $M \in \mathsf{SMod}^C$, the following are equivalent:
\begin{itemize}
\item[(1)]
$M$ is injective as a right $C$-comodule;
\item[(2)]
$M$ is an injective object in $\mathsf{SMod}^C$;
\item[(3)]
$M$ is coflat as a right $C$-comodule in the sense that the cotensor product functor 
$M \Box_C : {}^C\mathsf{Mod} \to \mathsf{Mod}_K$ is exact;
\item[(4)]
The cotensor product functor 
$M \Box_C : {}^C\mathsf{SMod} \to \mathsf{SMod}_K$ is exact.
\end{itemize}
A parallel result holds true for every object in ${}^C\mathsf{SMod}$.
\end{prop}

\begin{proof}
The equivalence (1) $\Leftrightarrow$ (3) is due to Takeuchi \cite{tak1}, Proposition A.2.1.
Note that if $N \in \mathsf{SMod}^C$ is finite-dimensional over $K$, the dual vector space $N^*$ is naturally
an object in ${}^C\mathsf{SMod}$, and we have 
$$ (M \Box_C N^*)_i = \mathrm{Hom}_{\mathsf{SMod}^C}(N, M[i]), \quad i = 0, 1, $$
where $M[0] = M$, and $M[1]$ is the degree shift of $M$. Then a slight modification of the proof of 
\cite{tak1}, Proposition A.2.1 shows that (4) $\Rightarrow$ (2). 

Obviously, (3) $\Rightarrow$ (4). The proof will complete if we prove (2) $\Rightarrow$ (1). 
Assume (2). Then the structure morphism $M \to M \otimes C$ in $\mathsf{SMod}^C$
splits. This implies (1), since $M \otimes C$ is an injective object in $\mathsf{Mod}^C$. 
\end{proof}

\section{$K$-functors}\label{K-functors}

In what follows we use definitions and notations from \cite{z}.  
Recall that $\mathsf{SAlg}_K$ is a category of supercommutative superalgebras over a field $K$ whose characteristic
is different from 2.
The category $\mathsf{Alg}_K$ of commutative $K$-algebras can be regarded as a full subcategory of $\mathsf{SAlg}_K$.
For simplicity we denote the functor category $\mathsf{Sets}^{\mathsf{SAlg}_K}$ as
\begin{equation*}
\mathcal{F} = \mathsf{Sets}^{\mathsf{SAlg}_K}.
\end{equation*}
An object in this category is called a \emph{$K$-functor}. 
For $A\in \mathsf{SAlg}_K$, denote the $K$-functor $h^{A}$ by $SSp \ A$, and call such a $K$-functor an {\it affine superscheme}. The superalgebra $A$
is called the {\it coordinate superalgebra} of $X=SSp \ A$, and is denoted by $K[X]$. 

Let $I$ be a superideal of $A\in \mathsf{SAlg}_K$. Define a {\it closed} subfunctor $V(I)$ and an {\it open}
subfunctor $D(I)$ of $SSp \ A$ as follows (cf. \cite{z}, p.719; see also \cite{jan}, Part I, 1.4 -- 1.5). For any $B\in\mathsf{SAlg}_K$ set
$$V(I)(B)=\{x\in SSp \ A(B) | x(I)=0\}$$
and 
$$
D(I)(B)=\{x\in SSp \ A(B) | x(I)B=B\}.$$
Since $V(I)\simeq SSp \ A/I$, we call
$Y=V(I)$ a {\it closed supersubscheme} of $SSp \ A$ and $I=I_Y$ the {\it defining ideal} of $Y$.
All standard properties of closed and open subfunctors of affine schemes mentioned in \cite{jan}, Part I, 1.4-1.5, are translated to the category of affine superschemes per verbatim; see also \cite{z}, Lemma 2.2. 

Let $X$ be a $K$-functor. A subfunctor $Y\subseteq X$ is said to be {\it closed} ({\it open}) iff for any morphism
${\bf f} : SSp \ A \to X$ in $\mathcal{F}$ the pre-image ${\bf f}^{-1}(Y)$ is closed (respectively, open) in $SSp \ A$. These definitions are copied from  \cite{jan}, Part I, 1.7 and 1.12 (or from \cite{dg},
I, \S 1, 3.6 and \S 2, 4.1). Again the properties of open and closed subfunctors mentioned in \cite{jan}, {Part I, 1.7, 1.12}, can be translated to the category $\mathcal{F}$
per verbatim (see Lemma 9.1 below). We call such a translation a {\it superization} of the corresponding property. Proofs of superizations that are not difficult are left to the reader. For example, let us remark a (super)variant of 1.7({6}) from \cite{jan}, Part I. 
If $Y$ is an open subfunctor of $X$ and if $\alpha : A\to A'$ is a morphism of superalgebras, then
$X(\alpha)^{-1}(Y(A'))=Y(A)$ whenever $A'$ is a faithfully flat $A$-supermodule via $\alpha$, or $A'_0$ is a faithfully 
flat $A_0$-module via $\alpha_0$ (cf. Lemma 1.3, \cite{z}).  

Given $R \in \mathsf{SAlg}_K$ and an $R$-superalgebra $A$, we let $\iota_R^A$ denote the canonical morphism 
$R\to A$. In particular,
$\iota_{A_0}^A$ denotes the inclusion $A_0 \hookrightarrow A$ from the even component $A_0$ into $A$. 
Let $X\in \mathcal{F}$. Set 
\begin{equation}\label{Xev}
X_{ev}(A)=X(\iota^A_{A_0})(X(A_0)),\ A \in \mathsf{SAlg}_K. 
\end{equation}
Obviously, $X_{ev}$ is a subfunctor of $X$, and $X \mapsto X_{ev}$ is an endofunctor of $\mathcal{F}$ preserving inclusions. More precisely, if $Y$ is a subfunctor of $X$, then $Y_{ev}\subseteq Y\bigcap X_{ev}$. 
Given $A \in \mathsf{SAlg}_K$, we set
\begin{equation}\label{overlineA}
\overline{A} := A/AA_1 = A_0/A_1^2. 
\end{equation}
This is the largest purely even quotient algebra of $A$. We see $$(SSp \ A)_{ev}=V(AA_1)\simeq SSp \ \overline{A}.$$ More generally, $$V(I)_{ev}=V(I)\bigcap (SSp \ A)_{ev}\simeq
V(I+AA_1).$$

\begin{lemma}\label{open}
The following statements hold:

1. If $I$ is a superideal of $A\in \mathsf{SAlg}_K$, then $D(I)_{ev}=D(I)\bigcap (SSp \ A)_{ev}$;

2. If $Y$ and $Y'$ are open subfunctors of a $K$-functor $X$, then $Y=Y'$ iff $Y(C)=Y'(C)$ for any $C\in \mathsf{F}_K$ iff $Y_{ev}=Y'_{ev}$.
\end{lemma}

\begin{proof}
Observe that $\alpha\in D(I)_{ev}(C)$ iff there are $x_1, \ldots , x_n\in I_0$ and $c_1, \ldots , c_n\in C_0$ such that
$\sum_{1\leq i\leq n}c_i\alpha(x_i)=1, C\in \mathsf{SAlg}_K$. The first statement follows by Lemma 1.3, \cite{z}. 
The proof of the first equivalence in the second statement can be copied from 1.7(4), \cite{jan}, Part I; see also  
\cite{z}, Lemma 1.1.   
It remains to notice that $Z(C)=Z_{ev}(C)$
for all $Z\in \mathcal{F}$ and $C \in \mathsf{Alg}_K$. 
\end{proof}

A collection of open subfunctors $\{Y_i\}_{i\in I}$ of a $K$-functor $X$ is called an {\it open covering} whenever $X(A)=\bigcup_{i\in I}
Y_i(A)$ for any $A\in \mathsf{F}_K$ (cf. \cite{jan}, Part I, 1.7). 

Let us understand through the natural identification that the symbol $(\mathsf{SAlg}_K)^{\circ}$, which denotes the 
opposite category of $\mathsf{SAlg}_K$, represents the category of affine superschemes over $K$.
We define a Grothendieck topology $T_{loc}$ in $(\mathsf{SAlg}_K)^{\circ}$ as follows.
A covering in $T_{loc}$ is defined to be a
collection of finitely many morphisms $\{SSp~R_{f_i}\to SSp~R\}_{1\leq i\leq n}$, 
where $R\in \mathsf{SAlg}_K, f_1, \ldots, f_n\in R_0$ such that
$\sum_{1\leq i\leq n}R_0 f_i=R_0$. As is easily seen, the thus defined coverings satisfy the conditions 1-3 given in \cite{bd}, p.46, so that 
$T_{loc}$ is indeed a Grothendieck topology.
Notice that each $SSp \ R_{f_i}\to SSp \ R$ is an isomorphism onto $D(Rf_i)$ and
$D(Rf_i)$ form an open covering of $SSp \ R$. 

\begin{definition}\label{def-superscheme}
1. A sheaf $X$ on $T_{loc}$ is called a \emph{local functor}. By the definition, $X \in \mathcal{F}$.
Observe that any affine superscheme is a local functor (cf. \cite{jan}, Part I, 1.8(4)). 

2. A local $K$-functor $X$ is called a \emph{superscheme} provided $X$ has an open covering
$\{Y_i\}_{i\in I}$ with $Y_i\simeq SSp \ A_i, A_i\in \mathsf{SAlg}_K$. 
The full subcategory of all superschemes in $\mathcal{F}$ is denoted by $\mathcal{SF}$.    

3. A superscheme $X$ is said to be \emph{Noetherian}, 
if it has an open covering $\{Y_i\}_{i\in I}$ with $Y_i\simeq SSp \ A_i$, as above,
such that $I$ is finite, and each $A_i$ is Noetherian; see Proposition \ref{Noetherian}. 
Observe that an affine superscheme $SSp~A$ is Noetherian iff $A$ is Noetherian.  
\end{definition}

\begin{prop}\label{local}
(cf. \cite{jan}, Part I, 1.8, or \cite{dg}, I, \S 1, Proposition 4.13)\
A $K$-functor $X$ is local iff for any $K$-functor $Y$ and its open covering $\{Y_i\}_{i\in I}$ the diagram
$$(*) \hspace{7mm} \mathrm{Mor}_{\mathcal{F}}(Y, X)\to\prod_{i\in I}\mathrm{Mor}_{\mathcal{F}}(Y_i, X)\begin{array}{c}\to \\
\to\end{array}\prod_{i, j\in I} \mathrm{Mor}_{\mathcal{F}}(Y_i\bigcap Y_j, X)$$
is exact. 
\end{prop}

\begin{proof}
The part $``$if" is by Yoneda's Lemma applied to $Y=SSp \ R$ and its open covering $\{Y_i=D(Rf_i)\}_{1\leq i\leq n}$. 
Conversely, assume that all diagrams (*) are exact, provided
$Y=SSp \ R$ and $Y_i=D(J_i), i\in I$ ($I$ is not necessary finite). 

Consider two morphisms ${\bf f}, {\bf g} : Y\to X$ such that ${\bf f}\neq {\bf g}$ but ${\bf f}|_{Y_i}=
{\bf g}|_{Y_i}$ for all $i\in I$. There is a superalgebra $A\in \mathsf{SAlg}_K$
and an element $y\in Y(A)$ with ${\bf f}(A)(y)\neq {\bf g}(A)(y)$. By Yoneda's
Lemma $y$ induces a morphism ${\bf g}_y : SSp \ A\to Y$ that satisfies ${\bf f}'={\bf f}
{\bf g}_y\neq {\bf g}'={\bf g} {\bf g}_y$. On the other hand, $\{ Y'_i={\bf g}_y^{-1}(Y_i) \}_{i\in
I}$ is an open covering of $SSp \ A$ and ${\bf f}'|_{Y'_i}={\bf g}'|_{Y'_i}$ for
all $i$. The contradiction implies that the map on the left is injective.

Suppose that $({\bf f}_i)_{i\in I}$ belongs to the kernel of the maps on the right. Let $\mathcal{P}$ be a set consisting of pairs $(T, {\bf g})$,
where $T$ is a subfunctor of $Y$ that contains each $Y_i$ and ${\bf g} \in \mathrm{Mor}_{\mathcal
{F}}(T, X)$ satisfies ${\bf g}|_{Y_i}={\bf f}_i,\ i\in I$. The set $\mathcal{P}$ is
not empty and it is partially ordered by $(T, {\bf g})\leq (T', {\bf g}')$ iff
$T\subseteq T'$ and ${\bf g}'|_T={\bf g}$. By Zorn's Lemma $\mathcal{P}$ contains a
maximal element $(T, {\bf g})$. Assume that $T\neq Y$. There exists $A\in
\mathsf{SAlg}_K$ and $y\in Y(A)\setminus T(A)$. As above, $y$ induces the
morphism ${\bf g}_y : SSp \ A\to Y$. By assumption, there exists a
unique morphism ${\bf u} : SSp \ A\to X$ that is defined by
${\bf u}|_{{\bf g}_y^{-1}(Y_i)}={\bf f}_i {\bf g}_y, i\in I$. It follows that
${\bf u}|_{{\bf g}_y^{-1}(T)}={\bf g}{\bf g}_y|_{{\bf g}_y^{-1}(T)}$. Define a pair $(T', {\bf g}') >
(T, {\bf g})$ as follows. For a superalgebra $B\in \mathsf{SAlg}_K$ set
$$T'(B)= T(B)\bigcup\{Y(\alpha)(y)|\alpha\in SSp \ A (B)\}
=T(B)\bigcup\{{\bf f}_y(\alpha)|\alpha\in SSp \ A (B)\}.$$ The morphism
${\bf g}'$ is defined by ${\bf g}'(B)|_{T(B)}={\bf g}(B)$ and
${\bf g}'(B)({\bf g}_y(\alpha))={\bf u}(B)(\alpha)$. It can be easily checked that
$T'$ is a subfunctor of $Y$, $T\subseteq T', T\neq T'$ and ${\bf g}'$ is
correctly defined. 

It remains to consider a diagram with $Y=SSp \ R$ and $Y_i=D(J_i), i\in I$. Replace $\{Y_i=D(J_i)\}_{i\in I}$ by the open covering $\{ D(Rf)\}_{f\in (J_i)_0 , i\in I}$. Observe that any open covering
of $Y=SSp \ R$ contains a finite subcovering.  In particular, for any finite subset
$S\subseteq \{(f, i)|f\in (J_i)_0, i\in I\}$ such that $\{D(Rf)\}_{(f, i)\in S}$ is an open covering of $Y$ there is 
a unique morphism ${\bf f}_S : Y\to X$ such that ${\bf f}_S |_{D(Rf)}={\bf f}_i |_{D(Rf)}, (f, i)\in S$. 
In particular, $S\subseteq S'$ implies ${\bf f}_S={\bf f}_{S'}$ and by the above injectivity, 
${\bf f}_S|_{D(J_i)}={\bf f}_i |_{D(J_i)}$ for each $i\in I$. 
Thus ${\bf f}={\bf f}_S$ is the required pre-image of $({\bf f}_i)_{i\in I}$. 
\end{proof}
Superizing \cite{jan}, Part I, 1.9 (2)) and 1.12 (4, 5, 6), we see that an open or closed subfunctor $Y$ of a local $K$-functor $X$ (respectively, of a superscheme $X$) is again
local (a superscheme).
 
\begin{lemma}\label{covering}
Let ${\bf f} : X\to Y$ be a morphism of local functors, and let $\{ Y_i\}_{i\in I}$ be an open covering of $Y$. Then
$\bf f$ is an isomorphism iff each ${\bf f}|_{{\bf f}^{-1}(Y_i)}$ is.
\end{lemma}

\begin{proof}
The part $``$if" is obvious. By Yoneda's Lemma one has to show that the map $\mathrm{Mor}_{\mathcal{F}}(SSp \ A, X)\to \mathrm{Mor}_{\mathcal{F}}(SSp \ A, Y)$, induced by
$\bf f$, is a bijection for any $A\in \mathsf{SAlg}_K$. It easily follows by Proposition \ref{local}. We leave details to the reader. 
\end{proof}   

To describe open affine supersubschemes of an affine superscheme $SSp~A$, let 
$\phi : A\to B$ be a morphism in $\mathsf{SAlg}_K$; it induces the morphism of superschemes $SSp \ \phi : SSp \ B\to SSp \ A$.

\begin{lemma}\label{openaffine}
$SSp \ \phi$ is an isomorphism of $SSp \ B$ onto an open subfunctor of $SSp \ A$ iff the following conditions hold: 
\begin{enumerate}
\item There are elements $x_1, \ldots , x_t\in A_0$ such that $\sum_{1\leq i\leq t}B_0\phi(x_i)=B_0$.
\item The induced morphisms $A_{x_i}\to B_{\phi(x_i)}$ are isomorphisms. 
\end{enumerate}
\end{lemma}

\begin{proof}
If $SSp \ \phi$ is an isomorphism of $SSp \ B$ onto $D(I)\subseteq SSp \ A$,
then $\phi\in D(I)(B)$. In other words, there are $x_1, \ldots , x_t\in I_0$ 
such that $\sum_{1\leq i\leq t}B_0\phi(x_i)=B_0$. We have $(SSp \ \phi)^{-1}(D(Ax))=D(B\phi(x)), x\in A_0$.
Since $D(Ax)\subseteq D(I)$ whenever $x\in I_0$, $SSp \ \phi$ induces the isomorphism $SSp \ B_{\phi(x)}\to SSp \ A_x$.
Conversely, assume that conditions 1 and 2 hold. It is clear that $SSp \ \phi$ induces a morphism $SSp \ B\to D(I), I=\sum_{1\leq i\leq t} Ax_i$. 
It remains to notice that $D(Ax_i)$ and $D(B\phi(x_i))$ form open coverings of $D(I)$ and $SSp \ B$ respectively. 
Moreover, $SSp \ \phi$ induces an isomorphism $D(B\phi(x_i))$ onto $D(Ax_i), 1\leq i\leq t$. Lemma \ref{covering}  concludes the proof.
\end{proof}

Let us define a Grothendieck topology $T_{fppf}$ of {\it fppf coverings} in $(\mathsf{SAlg}_K)^{\circ}$. 
A covering in $T_{fppf}$ is defined to be a collection of finitely many morphisms
$\{SSp~R_i\to SSp~R\}_{1\leq i\leq n}$, where each $R_i$ is a finitely presented $R$-superalgebra 
and $R_1\times\ldots\times R_n$ is a faithfully flat $R$-module. A sheaf on $T_{fppf}$ is called a {\it $K$-sheaf} 
(or {\it faisceau} in the terminology from \cite{dg}, III, \S 1).

\begin{prop}\label{sheafification}
For any $X\in \mathcal{F}$ there are a $K$-sheaf $\tilde{X}$, called
{\it sheafification} of $X$, and a morphism $p : X\to \tilde{X}$
such that for any $K$-sheaf $Y$ the canonical map $\mathrm{Mor}(\tilde{X},
Y)\to \mathrm{Mor}(X, Y)$ induced by $p$ is a bijection.
\end{prop}

This proposition is a special case of more general statement about
sheafifications of functors on sites; see \cite{vist}, Theorem 2.64.

\begin{definition}\label{def-suitable}
For $R, R' \in \mathsf{SAlg}_K$, we write $R \leq R'$ or $R'\geq R$, if $R'$ is a fppf covering of $R$.
A $K$-functor $X$ is called \emph{suitable} if it commutes with finite direct products of
superalgebras, and if given $R \leq R'$,
the induced map $X(R)\to X(R')$ is injective.
\end{definition}

\begin{rem}\label{suitable} 
(cf. \cite{z}, pp.721-722, or \cite{jan}, Part I, 5.4)\
If $X$ is suitable, then for any $A\in \mathsf{SAlg}_K$ $$\tilde{X}(A)=\lim\limits_{\rightarrow} X(B, A), X(B, A)=\mathrm{Ker}(X(B)
\begin{array}{c}\to \\
\to\end{array} X(B\otimes_A B)),$$ where 
$B$ runs over all fppf coverings of $A$. Besides, $p : X \to \tilde{X}$ is an
injection. For example, if $X$ is a subfunctor of a $K$-sheaf $Y$ that commutes with finite direct
products of superalgebras, then for each $R \in \mathsf{SAlg}_K$,
$$\tilde{X}(R)=\{y\in Y(R)|\mbox{there is} \ R'\geq R \
\mbox{such that} \ Y(\iota_R^{R'})(y)\in X(R')\}$$
\end{rem}

Let $G$ be a group $K$-sheaf, and let $X$ be a suitable $K$-functor on which $G$ acts on the right.
Assume that $G$ acts on $X$ \emph{freely}, so that for any $R\in\mathsf{SAlg}_K$, 
the group $G(R)$ acts freely on $X(R)$. 
The functor $A\mapsto X(A)/G(A), A\in \mathsf{SAlg}_K,$ is called the {\it naive quotient} of $X$
over $G$, and is denoted by $(X/G)_{(n)}$. As in \cite{jan}, Part I, 5.5, one can check that
$(X/G)_{(n)}$ is a suitable $K$-functor. Its sheafification is called the {\it quotient $K$-sheaf}
of $X$ over $G$, and is denoted by $X \tilde{/} G$. By Remark \ref{suitable}, we have 
$(X/G)_{(n)} \subseteq X \tilde{/} G$ (cf. \cite{z}, pp. 725-726). Symmetrically one can define 
a quotient $K$-sheaf of $X$ over $G$, provided $G$ acts freely on $X$ on the left.
 
By an \emph{algebraic supergroup} we always mean an algebraic affine supergroup. Let $G$ be such 
a supergroup.  Thus, $G=SSp \ K[G]$, where $K[G]$ is 
a Hopf superalgebra that is finitely generated as an algebra.
Let $H$ be a closed supersubgroup of $G$, that is $H=V(I_H)$, where $I_H$ is a Hopf superideal of $K[G]$.
Since $H$ acts freely on $G$ on the right one can define the quotient $K$-sheaf $G \tilde{/} H$.

Another Grothendieck topology that is, however, less concerned with in this paper is the Grothendieck topology $T_{ff}$ of 
\emph{faithfully flat coverings} in $(\mathsf{SAlg}_K)^{\circ}$, which associates to $SSp~R$, collections of finitely many morphisms
$\{SSp~R_i\to SSp~R\}_{1\leq i\leq n}$ such that $R_1\times\ldots\times R_n$ is a faithfully flat $R$-module. 
A sheaf on $T_{ff}$ is called a {\it dur $K$-sheaf}.
Since $T_{fppf} \subseteq T_{ff}$, a dur $K$-sheaf is necessarily a $K$-sheaf. 
In parallel to Proposition \ref{sheafification},
it is known that for any $X\in \mathcal{F}$, there exist uniquely a dur $K$-sheaf $\tilde{\tilde{X}}$ together
with a morphism $X\to \tilde{\tilde{X}}$ which have the obvious universal property. This $\tilde{\tilde{X}}$ is called the
\emph{dur sheafification} of $X$. The dur sheafification of the naive quotient $(X/G)_{(n)}$ as defined above is denoted by 
$X \tilde{\tilde{/}} G$. 

\begin{rem}\label{sheaf is local}
By Lemma 1.2 (ii), \cite{z}, $T_{loc} \subseteq T_{fppf}$ and $T_{loc} \subseteq T_{ff}$ as well. Thus any $K$-sheaf (or dur $K$-sheaf) is a local $K$-functor.
\end{rem}
\section{Geometric superspaces}\label{superspaces}

A {\it geometric superspace} $X$ consists of a topological space $X^e$ and a sheaf
of commutative superalgebras $\mathcal{O}_X$ such that all stalks $\mathcal{O}_x,\ x \in X^e$, are local superalgebras; see below.
A morphism of superspaces $f : X\to Y$ is a pair $(f^e, f^*)$, where $f^e : X^e\to
Y^e$ is a morphism of topological spaces and $f^* : \mathcal{O}_Y\to
f^e_*\mathcal{O}_X$ is a morphism of sheaves such that $f^*_x : \mathcal{O}_{Y,
f(x)}\to \mathcal{O}_{X, x}$ is a local morphism for any $x\in X^e$.  
We let  $\mathcal{V}$ denote the category of superspaces. 

Let $A\in \mathsf{SAlg}_K$. By a {\it prime} (resp., {\it maximal}) {\it superideal} of $A$, we mean a super ideal of the form 
$p_0 \oplus A_1$, where $p_0$ is a prime (resp., maximal) ideal of $A_0$. A prime super ideal is characterized as a two-sided
ideal $p$ of $A$ such that $A/p$ is an integral domain. A maximal superideal is the same as a maximal left (or right) ideal of $A$.
It follows that the Jacobson radical $rad \, A$ of $A$ equals $rad \, A_0 \oplus A_1$; see \cite{z}, Lemma 1.1. 
A superalgebra with a unique maximal superideal is said to be 
\emph{local}. 
We define the {\it localization} $A_p$ of $A$ at a prime superideal $p = p_0 \oplus A_1$ by $A_p = (A_0 \setminus p_0)^{-1} A$;
this is a local superalgebra. 
A morphism $\alpha : A \to B$ between local superalgebras
is said to be \emph{local} if $\alpha(\mathfrak{m})\subseteq\mathfrak{n}$, where $\mathfrak{m}$ and $\mathfrak{n}$ are 
the unique maximal superideals of
$A$ and $B$, respectively.

We define an {\it affine superspace} $SSpec \ A$ as follows. The
underlying topological space of $SSpec \ A$ coincides with the prime spectrum (that is, the set of all primes) 
of $A$ endowed with the Zariski topology. In other words, $U\subseteq (SSpec \ A)^e$ is open iff
there is a superideal $I$ in $A$ such that $U=U(I)=\{p\in (SSpec \ A)^e| I\not\subseteq p\}$. For any open subset
$U \subseteq (SSpec \ A)^e$ the superalgebra $\mathcal{O}_{SSpec \
A}(U)$ consists of all locally constant functions $h : U\to
\bigsqcup_{p\in U} A_p$ such that $h(p)\in A_p$, $p \in U$. 

A superspace $X$ is called a {\it geometric
superscheme} iff there is an open covering $X^e=\bigcup_{i\in I}
U_i$ such that $(U_i, \mathcal{O}_X |_{U_i})\simeq SSpec \ A_i,
A_i\in \mathsf{SAlg}_K, i\in I$. We let $\mathcal{SV}$ denote the full subcategory of all geometric
superschemes in the category $\mathcal{V}$ of superspaces.

\begin{lemma}\label{canobijection} 
(\cite{dg}, I, \S 1, Theorem 2.1)\ 
There is a canonical bijection 
$$\mathrm{Mor}_{\mathcal{V}}(X, SSpec\ A) \simeq \mathrm{Hom}_{\mathsf{SAlg}_K}(A, \mathcal{O}_X(X)),$$ 
which is functorial in both arguments.
\end{lemma}

\begin{proof}
A morphism $f : X\to SSpec \ A$
induces $\phi(f)=f^*(X) : A =\mathcal{O}_{SSpec \ A}(SSpec \
A) \to \mathcal{O}_X(X)$. Conversely, let $\phi : A \to \mathcal{O}_X(X)$ be a superalgebra morphism.
Let $\phi_x$ denote the composite of $\phi$ with the canonical {\it evaluation} morphism $\mathcal{O}_X(X)\to \mathcal{O}_x , x\in X$.
Define $f^e(\phi) : X^e \to (SSpec \ A)^e$ by $f^e(\phi)(x)=\phi^{-1}(m_x)$,
where $x\in X$ and $m_x$ is the maximal superideal of the local superalgebra 
$\mathcal{O}_x$. The corresponding morphism of sheaves $f(\phi)^*$ can be defined
locally. If $f^e(\phi)(x)=p$, then $f(\phi)^*(\frac{a}{b})=\frac{\phi(a)}{\phi(b)}, \frac{a}{b}\in
A_p$.
\end{proof}
Set $X=SSpec \ B$ in Lemma \ref{canobijection}. Then $\mathrm{Mor}_{\mathcal{V}}(SSpec \ B, SSpec \
A)\simeq \mathrm{Hom}_{\mathsf{SAlg}_K}(A, B)$. A morphism induced by $\phi\in \mathrm{Hom}_{\mathsf{SAlg}_K}(A, B)$ is denoted by $SSpec \ \phi$.

Let $X$ be a geometric superspace, and suppose $f\in {\mathcal O}_X(X^e)_0$. Define a subset $X^e_f$ as follows. A point $x\in X^e$ belongs to
$X^e_f$ iff $f$ is an invertible element in ${\mathcal O}_x$. For example, if $X=SSpec \ A$, then $X_f=U(Af)$. 
\begin{lemma}\label{principalopen}
The following statements hold: \\
1. If $g : X\to Y$ is a morphism of geometric superspaces and $h\in {\mathcal O}_Y(Y^e)$, then $g^{-1}(Y^e_h)=X^e_{g^*(h)}$. \\
2. If $X$ is a geometric superscheme, then $X^e_f$ is open for any $f\in {\mathcal O}_X(X^e)$. In particular,
${\mathcal O}_X(X^e_f)\simeq {\mathcal O}_X(X^e)_f$.
\end{lemma}  
\begin{proof}
The first statement is obvious. If $U\simeq SSpec \ R$ is an open subspace of $X$, then one can apply the first statement for the inclusion
$U\to X$.
\end{proof}
From now on we denote the open subspace $(X^e_f, {\mathcal O}_X|_{X^e_f})$ by $X_f$.
\begin{lemma}\label{limitexist}
The categories $\mathcal{V}$ and $\mathcal{F}$ are closed with respect to direct sums
and cokernels of morphisms.  In particular,
any (covariant) functor from a small category to 
$\mathcal{V}$ (respectively, to $\mathcal{F}$) has a direct limit.
\end{lemma}
\begin{proof}
The second statement follows by Proposition 3.4,
\cite{bd}. The proof of the first statement for $\mathcal{V}$ can be copied from \cite{dg}, I, \S 1, Proposition 1.6. 
Since the category $\mathsf{Sets}$ is closed with respect to direct products and cokernels of morphisms (cf. \cite{bd}), the first statement for $\mathcal{F}$ follows. 
\end{proof}
 
\section{Comparison Theorem}\label{sec-comparison}

Let $X\in \mathcal{F}$. Recall from Section 1 the definition of the category $\mathcal{M}_X$.
Define the functor $d_X : (\mathcal{M}_X)^{\circ} \to
\mathcal{V}$ by $d_X(R, x) =SSpec \ R, d_X(\phi)=SSpec \ \phi$. By Lemma \ref{limitexist} $|X|= \lim\limits_{\rightarrow} d_X$ exists and belongs to $\mathcal{V}$.
The geometric superspace $|X|$ is called a {\it geometric realization} of $X$ (cf. \cite{dg}, I, \S 1, 4.2).
Besides, $X \mapsto |X|$ is a functor from $\mathcal{F}$ to $\mathcal{V}$. If ${\bf f} : X \to Y$ is a morphism in $\mathcal{F}$, then 
$|{\bf f}| : |X|\to |Y|$ is (uniquely) defined by  the collection of morphisms $\{i_{R, {\bf f}(R)(x)}\}_{(R, x)\in \mathcal{M}_X}$.
For the sake of simplicity the canonical morphisms $i_{(R, x)} : d_X(R, x) \to |X|$ are denoted just by $i_x$. 

\begin{example}\label{affinespace}
$SSpec \ A\simeq |SSp \ A|$. Indeed, the collection of morphisms $SSpec \ \phi : SSpec \ R\to SSpec \ A, \phi\in SSp \ A(R),$
satisfies the first condition. For any other collection $i'_{\phi} : SSpec \ R\to X$ as in the second condition, set $g=i'_{\mathrm{id}_A}$.
Moreover, Yoneda's Lemma implies that
$|{\bf g}_x|=i_x$ for any $x\in X(A)$.  
\end{example}

Define a functor $\mathcal{V} \to \mathcal{F}, X  \mapsto X^{\diamond}$, where
$X^{\diamond}(R)=\mathrm{Mor}_{\mathcal{V}}(SSpec \ R , X), R\in \mathsf{SAlg}_K$. In other words, $X^{\diamond}$ is a restriction
of $h_X$ on the full subcategory of affine superspaces ({\it functor of points} in the terminology of \cite{cf}).
Notice that $(SSpec \ A)^{\diamond}=SSp \ A$ by Lemma \ref{canobijection}.
 
\begin{lemma}\label{diamond} 
The following statements hold: 

1. For any $X\in \mathcal{V}$ and its open supersubspace $U$, $U^{\diamond}$ is an open subfunctor of $X^{\diamond}$; 

2. If $\{U_i\}_{i\in I}$ is an open covering of $X$, then
$\{ U^{\diamond}_i \}_{i\in I}$ is an open covering of $X^{\diamond}$; 

3. $(SSpec \ A)^{\diamond}$ and $(SSpec \ \phi)^{\diamond}$
are identified with $SSp \ A$ and $SSp \ \phi$ respectively, where $\phi : A\to B$; 

4. If $I$ is a superideal in $A$, then $U(I)^{\diamond}=D(I)$.
\end{lemma}

\begin{proof}
Assume that $x\in X^{\diamond}(R)$ defines a morphism
${\bf g}_x : SSp \ R\to X^{\diamond}$. We have $x^{-1}(U)=U(I)$ for 
a superideal $I$ in $R$. Observe that $\phi\in SSp \ R(B)$ belongs to ${\bf g}_x^{-1}(U^{\diamond})(B)$ iff for any $p\in (SSpec \ B)^e$ $I\not\subseteq \phi^{-1}(p)$. Thus  ${\bf g}_x^{-1}(U^{\diamond})=D(I)$. The last statements are left for the reader.
\end{proof}

\begin{prop}\label{adjoint} 
(\cite{dg}, I, \S 1, Proposition 4.1)
The functor $X \mapsto |X|$ is a left adjoint to
$Y \mapsto  Y^{\diamond}$.
\end{prop}

\begin{proof}
Let $f\in \mathrm{Mor}_{\mathcal{V}}(|X|, Y)$. Define ${\bf f} : X\to Y^{\diamond}$ by ${\bf f}(R)(x)=fi_x, x\in X(R)$. 
By a routine verification ${\bf f}$ is a morphism of functors. Moreover, $f\mapsto {\bf f}$ depend of both $X$ and $Y$ functorially.
Since $\lim\limits_{\rightarrow} \delta_X=X$ and all diagrams 
$$\begin{array}{ccc}
\mathrm{Mor}_{\mathcal{V}}(|X|, Y) & \to & \mathrm{Mor}_{\mathcal{F}}(X, Y^{\diamond}) \\
h_Y(i_x) \downarrow & & \downarrow h_{Y^{\diamond}}({\bf g}_x)\\
\mathrm{Mor}_{\mathcal{V}}(SSpec \ R, Y) & \to & \mathrm{Mor}_{\mathcal{F}}(SSp \ R,
Y^{\diamond})
\end{array}$$
are commutative, where the bottom arrows are identical maps, the proof follows then 
by applying Lemma \ref{prlimit}.
\end{proof}

Let $A\in \mathsf{SAlg}_K, B\in \mathsf{F}_K$; see Example \ref{field} for the definition of $\mathsf{F}_K$.
In what follows, $SSp \ A(B)$ is considered as a topological space whose open subsets are $D(I)(B)$, where
$I$ runs over all superideals of $A$; see \cite{jan}, Part I, 1.4(5) and 1.5(8). 

\begin{example}\label{F_K}
(\cite{dg}, I, \S 1, 4.8)\
One sees that $f=SSp \ A|_{\mathsf{F}_K}$ is a functor from $\mathsf{F}_K$ to the category of topological spaces. 
Note that $\lim\limits_{\rightarrow} f$ and $(SSpec \ A)^e$ are identified as topological spaces. 
By combining Lemma \ref{limit} with Example \ref{field}, we see that $a\in f(B), b\in f(C)$ are equivalent iff 
$\mathrm{Ker}~a=\mathrm{Ker}~b=p\in (SSpec \ A)^e$. Thus, $\lim\limits_{\rightarrow} f$ and $(SSpec \ A)^e$ 
are identified as sets. 
A subset $V\subseteq \lim\limits_{\rightarrow} f$ is open iff $i_B^{-1}(V)$ is open in $SSp \ A(B)$ for any 
$B\in \mathsf{F}_K$; see Remark \ref{commdiagram} below). 
If $V\subseteq \lim\limits_{\rightarrow} f$ is open, then $V=U(I)$, where $I=I_C$ and $C$ is a compositum of 
the fields $Q(A/p), p\in (SSpec \ A)^e$. Conversely, if $V=U(I)$, then $i_B^{-1}(V)=D(I_B)(B)$. 
Here $I_B$ is generated by $I\setminus p$, where $p\in U(I)$ and the field of fractions $Q(A/p)$ is isomorphic 
to a subfield of $B$.  
\end{example}

The following lemma is a refinement of \cite{dg}, I, \S 1, 4.9.

\begin{lemma}\label{natiso}
For any $X\in \mathcal{F}$ the topological space $|X|^e$ is naturally isomorphic to $\lim\limits_{\rightarrow} X|_{\mathsf{F}_K}$.
\end{lemma}

\begin{proof}
We have 
$$|X|^e=\lim\limits_{\stackrel{\rightarrow}{R, x}} (SSpec \ R)^e\simeq
\lim\limits_{\stackrel{\rightarrow}{R,
x}}\lim\limits_{\stackrel{\rightarrow}{B\in \mathsf{F}_K}}(SSp \
R)(B)\simeq$$$$\simeq\lim\limits_{\stackrel{\rightarrow}{B\in
\mathsf{F}_K}}\lim\limits_{\stackrel{\rightarrow}{R, x}}(SSp \ R)(B)\simeq
\lim\limits_{\stackrel{\rightarrow}{B\in \mathsf{F}_K}}X(B)\simeq
\lim\limits_{\rightarrow} X|_{\mathsf{F}_K}.$$
By Example \ref{F_K} each $X(B)$ has a structure of a topological space as follows. A subset $V\subseteq X(B)$ is open iff 
for any pair $(R, x)\in \mathcal{M}_X$ there is a superideal $I_x$ in $R$ such that ${\bf g}_x(B)^{-1}(V)=D(I_x)(B)$.
In its turn, $U\subseteq \lim\limits_{\rightarrow} X|_{\mathsf{F}_K}$ is open iff $i_B^{-1}(U)$ is open in $X(B)$ for any $B\in \mathsf{F}_K$.
\end{proof}

If ${\bf f} : X \to Y$ is a morphism in $\mathcal{F}$, then by Lemma \ref{natiso}
$|{\bf f}|^e$ is identified with $\lim\limits_{\rightarrow} {\bf f}|_{\mathsf{F}_K}$.
 
For a subset $P\subseteq |X|^e$ define a subfunctor $X_P$ of $X$ by $x\in X_P(A)\subseteq X(A)$ iff  $X(\phi)(x)\in\bigcup_{t\in P}
(X|_{\mathsf{F}_K})_t(B)=i_B^{-1}(P)$ for all $B\in \mathsf{F}_K, \phi\in SSp \ A(B)$.   
As in \cite{dg}, I, \S 1, 4.10, we have $X_P|_{\mathsf{F}_K}=\bigsqcup_{t\in P}(X|_{\mathsf{F}_K})_t$ and $\lim\limits_{\rightarrow} X_P|_{\mathsf{F}_K}=P$.

For any morphism of $K$-functors ${\bf f} : X\to Y$ and a subset $P\subseteq |Y|^e$ set $Q=(|{\bf f}|^e)^{-1}(P)$. Then 
$X_Q={\bf f}^{-1}(Y_P)$ (just observe that ${\bf f}$ takes $(X|_{\mathsf{F}_K})_t$ to $(Y|_{\mathsf{F}_K})_{|{\bf f}|^e(t)}, t\in |X|^e)$.

\begin{example}\label{residuefield} 
(\cite{dg}, I, \S 1, 4.11)\
Let $X=(X^e, \mathcal{O}_X)$ be a geometric superspace. Set $G=X^{\diamond}|_{\mathsf{F}_K}$. We have 
$$G(A)=\mathrm{Mor}_{\mathcal{V}}(SSpec \ A, X)= \bigsqcup_{x \in X^e}\mathrm{Hom}_{\mathsf{F}_K}(F(x), A),$$ 
where $A\in \mathsf{F}_K,\ F(x)= \mathcal{O}_{X, x}/m_x$. In other words, $|X^{\diamond}|^e$ is identified with $X^e$ as a set
and for any $x\in X^e$ we have $G_x=h^{F(x)}$.
\end{example}

Let $P\subseteq X^e, X\in \mathcal{V}$. Consider $P=P^e$ as a topological space with the induced topology and denote 
the natural embedding $P^e\to X^e$ by $i^e$. 
Set $\tilde{P}=(P^e, \mathcal{O}_P)$, where $\mathcal{O}_P=(i^e)^{-1}\mathcal{O}_X$ is the inverse image of $\mathcal{O}_X$; see \cite{hart}, II, \S 1. $\tilde{P}$ is a geometric superspace and we have the natural morphism of superspaces 
$\tilde{i}=(i^e, i^*) : \tilde{P}\to X$ that induces  isomorphisms $\mathcal{O}_{X, t}\simeq \mathcal{O}_{P, t}, t\in P$. 
Observe that any morphism of superspaces $f : Y\to X$ such that $f^e(Y^e)\subseteq P=P^e$ (uniquely) factors through $\tilde{i} : \tilde{P}\to X$.
In particular, the natural map $\mathrm{Mor}_{\mathcal{V}}(Y, \tilde{P})\to \mathrm{Mor}_{\mathcal{V}}(Y, X)$ is injective. 

We define the subfunctor $X^{\diamond}_P\subseteq X^{\diamond}$ as above.
By Example \ref{residuefield}, $X^{\diamond}_P(A)=\{x\in \mathrm{Mor}_{\mathcal{V}}(SSpec \ A, X)| x(SSpec \ A)\subseteq P^e\}=\tilde{P}^{\diamond}(A), A\in \mathsf{SAlg}_K$.  
By Lemma \ref{diamond} (i) $X^{\diamond}_P$ is an open subfunctor, provided $P^e$ is an open subset. 

\begin{example}\label{opensubset}
(\cite{dg}, I, \S 1, 4.11)\
A subset $P\subseteq (SSpec \ A)^e =|SSp \ A|^e$ is open iff $(SSp \ A)_P$ is an open subfunctor. The part $``$if" follows by
$(SSp \ A)_P=(SSpec \ A)^{\diamond}_P$. If $(SSp \ A)_P=D(I)$ for a superideal $I$ in $A$, then $$D(I)(B)=
\bigsqcup_{p\in P} \mathrm{Hom}_{\mathsf{F}_K}(Q(A/p), B)$$ 
for any field $B$. Thus $P=U(I)$.  
\end{example}

\begin{rem}\label{commdiagram}  
Let $X\in \mathcal{F}$. Then $P\subseteq |X|^e$ is open iff for any $(R, x)\in \mathcal{M}_X$ the pre-image $(i^e_x)^{-1}(P)$ is open in $(SSpec \ R)^e$. 
In fact, refine the topology on $|X|^e$ as follows. A subset $Q\subseteq |X|^e$ is claimed to be open iff $(i^e_x)^{-1}(Q)$ is open in
$(SSpec \ R)^e$ for any $(R, x)\in \mathcal{M}_X$. The set $|X|^e$ equipped with this topology is denoted by $X^e$. We have the morphism
$Y=(Y^e , \mathcal{O}_Y)\to |X|$, where $\mathcal{O}_Y$ is a sheafification of the pre-sheaf $\mathcal{O}_{|X|}$ on $Y^e$.
Moreover, each $i_x$ can be extended to a morphism $i'_x : SSpec \ R\to Y$ so that the diagrams
$$\begin{array}{ccc}
Y & \to & |X| \\
i'_x \nwarrow &  & \nearrow \ i_x \\
 & SSpec \ R &
\end{array}
$$
and
$$
\begin{array}{ccc}
 & Y & \\
i'_x \ \nearrow & & \nwarrow \ i'_y  \\
SSpec \ R & \rightarrow & SSpec \ S
\end{array}
$$
are commutative. Thus $Y\simeq |X|$ as geometric superspaces. 
In particular, $Y^e=|X|^e$ as topological spaces.
\end{rem}

Given $X \in \mathcal{F}$, let
$Op(X),\ Op(|X|^e)$
denote the set of all open subfunctors of $X$, and the set of all open subset of $Op(|X|^e)$, respectively.

\begin{prop}\label{open-open}
(\cite{dg}, I, \S 1, Proposition 4.12)\
Let $X \in \mathcal{F}$. Then the map 
$$Op(|X|^e) \to Op(X)$$
given by $P \mapsto X_P$ is a bijection.
\end{prop}

\begin{proof}
By Remark \ref{commdiagram} $P\subseteq |X|^e$ is open iff $Q=(i^e_x)^{-1}(P)$ is open in $SSpec \ R$ for any $R\in \mathsf{SAlg}_K, x\in X(R)$.  
By Examples \ref{affinespace} and \ref{opensubset}, all $Q$ are open iff all $(SSp \ R)_Q={\bf g}_x^{-1}(X_P)$ are open iff $X_P$ is open. Since $\lim\limits_{\rightarrow} X_P|_{\mathsf{F}_K}=P$
the map $P\mapsto X_P$ is injective. 

If $Y \in Op(X)$, then the set $\bigsqcup_{B\in \mathsf{F}_K}Y(B)$ is saturated with respect 
to the equivalence relation defining $\lim\limits_{\rightarrow} X|_{\mathsf{F}_K}$ (by superization of \cite{jan}, Part I, 1.7(6)). Set $P=
\lim\limits_{\rightarrow} Y|_{\mathsf{F}_K}$.
It follows that $X_P(A)$ consists of all $x\in X(A)$ such that for any
$\phi : A\to B, B\in \mathsf{F}_K,$ $X(\phi)(x)\in Y(B)$. In other words, for the morphism ${\bf g}_x : SSp \ A\to X$ and for the open subfunctor
$D(I)={\bf f}_x^{-1}(Y)$ any prime superideal $p\in (SSpec \ A)^e$ does not contain $I$. It immediately infers that $I=A$ and $x\in Y(A)$, that is
$X_P=Y$. 
\end{proof}
\begin{lemma}\label{fullfaithfunc}
The functor $X\to X^{\diamond}$ induces a full and faithful functor from $\mathcal{SV}$ to $\mathcal{SF}$.
\end{lemma}
\begin{proof} We have the diagram
\begin{equation*}
\begin{array}{cccc}
\mathrm{Mor}_{\mathcal{SV}}(X, Y) & \to & \prod_{j \in J} \mathrm{Mor}_{\mathcal{SV}}(X_j, Y)
\begin{array}{c} \to \\ \to \end{array} & \prod_{j, j' \in J} \mathrm{Mor}_{\mathcal{SV}}(X_j\cap X_{j'}, Y) \\
\downarrow & & \downarrow & \downarrow \\ \mathrm{Mor}_{\mathcal{SF}}(X^{\diamond}, Y^{\diamond}) & \to &
\prod_{j \in J} \mathrm{Mor}_{\mathcal{SF}}(X_j^{\diamond},
Y^{\diamond})
\begin{array}{c} \to \\ \to \end{array} & \prod_{j, j' \in J} \mathrm{Mor}_{\mathcal{
SF}}(X_j^{\diamond}\cap X_{j'}^{\diamond}, Y^{\diamond})
\end{array},
\end{equation*}
where $\{X_j\}_{j\in J}$ is an open affine covering of 
$X$. Since horizontal lines are exact sequences, it remains to consider
the case $X\simeq U(I)\subseteq SSpec \ A$. In its turn,
$U(I)$ has an affine covering $\{U(Af)\simeq SSpec \ A_f\}_{f\in I_0}$ and
$U(Af)\bigcap U(Ag)=U(Afg)$. In other words, the case $X=SSpec \ A$ is only needed and we see that
$\mathrm{Mor}_{\mathcal{SV}}(X, Y)=Y^{\diamond}(A)=\mathrm{Mor}_{\mathcal{SF}}(X^{\diamond},
Y^{\diamond})$. Proposition \ref{local} and Lemma \ref{diamond} imply $X^{\diamond}\in \mathcal{SF}$ whenever
$X\in \mathcal{SV}$. 
\end{proof}

The $m|n$-{\it affine superspace} ${\bf A}^{m|n}$ is defined by 
${\bf A}^{m|n}(B)=B_0^m\bigoplus B_1^n$ for $B\in\mathsf{SAlg}_K$; cf. \cite{z}, p.719. 
For any $K$-functor $X$, we denote $\mathrm{Mor}_{\mathcal F}(X, {\bf A}^{1|1})$ by ${\mathcal O}(X)$,  
following the notation of \cite{dg}; this is denoted by $K[X]$ in \cite{jan}. 
${\mathcal O}(X)$ has an obvious structure 
of a superalgebra; see \cite{dg}, I, \S 1, 3.3, or \cite{jan}, Part I, 1.3.
\begin{rem}\label{affinesheaf}
Let $X=SSp \ R$ and $D(Rf)$ is an open subfunctor of $X$. Then ${\mathcal O}(D(Rf))\simeq R_f\simeq {\mathcal O}_{SSpec \ R}(U(Rf))$. 
\end{rem}
\begin{prop}\label{sheaf}
Let $X\in {\mathcal F}$. The sheaf ${\mathcal O}_{|X|}$ is naturally isomorphic to $P\mapsto {\mathcal O}(X_P)$, where $P$ runs over open subsets of
$|X|^e$.
\end{prop}
\begin{proof}
If $P$ is a union of open subsets $P_i, i\in I$, then for any $B\in \mathsf{F}_K $ $X_P(B)=i_B^{-1}(P)=\bigcup_{i\in I}i_B^{-1}(P_i)=
\bigcup_{i\in I} X_{P_i}(B)$. Since ${\bf A}^{1|1}$ is an affine superscheme, Proposition \ref{local} implies that $P\mapsto {\mathcal O}(X_P)$
is a sheaf. By Remark \ref{affinesheaf} the statement holds for $X=SSp \ R$.

For each pair $(R, x)\in {\mathcal M}_X$ denote $i_x^{-1}(P)\subseteq (SSpec \ R)^e$ by $V_{R, x}$. We omit
the subindex $(R, x)$ if it does not lead to any confusion. Observe that $(SSp \ R)_V={\bf g}_x^{-1}(X_P)$; see the final notice in Example \ref{affinespace} and the notice before Example \ref{residuefield}.
 
By Lemmas \ref{restriction} and \ref{fiberprodset} we have
$$X_P(A)\simeq X_P(A)\times_{X(A)} X(A)=X_P(A)\times_{X(A)}(\lim\limits_{\rightarrow}\delta_X)(A)$$
$$\simeq
X_P(A)\times_{X(A)}\lim\limits_{\stackrel{\rightarrow}{R, x}} SSp \ R(A)\simeq\lim\limits_{\stackrel{\rightarrow}{R, x}} X_P(A)\times_{X(A)}
SSp \ R(A)$$$$\simeq\lim\limits_{\stackrel{\rightarrow}{R, x}}(SSp \ R)_V(A)\simeq (\lim\limits_{\stackrel{\rightarrow}{R, x}}(SSp \ R)_V)(A)$$
for any $A\in\mathsf{SAlg}_K$. Thus $X_P\simeq \lim\limits_{\stackrel{\rightarrow}{R, x}}(SSp \ R)_V$.
It remains to mimic the proof of \cite{dg}, I, \S 1, Proposition 4.14.  
\end{proof}
Proposition \ref{sheaf} infers that the functor $X\mapsto |X|$ induces the functor $\mathcal{SF}\to \mathcal{SV}$ which takes open subfunctors to
open subspaces.
\begin{lemma}\label{fullfaithspace}
The above functor $\mathcal{SF}\to \mathcal{SV}$ is full and faithful.
\end{lemma} 
\begin{proof}
Let $X, Y\in \mathcal{SF}$. If $\{X_i\}_{i\in I}$ is an open affine covering of $X$, then $|X_i|$ form an open affine covering of $X$.
As in Lemma \ref{fullfaithfunc}, we have a commutative diagram with exact horizontal lines:
\begin{equation*}
\begin{array}{cccc}
\mathrm{Mor}_{\mathcal{SF}}(X, Y) & \to & \prod_{i \in I}\mathrm{Mor}_{\mathcal{SF}}(X_i, Y)
\begin{array}{c} \to \\
 \to \end{array} & \prod_{i, i' \in I} \mathrm{Mor}_{\mathcal{SF}}(X_i \cap X_{i'}, Y) \\
\downarrow & & \downarrow & \downarrow \\ \mathrm{Mor}_{\mathcal{SV}}(|X|, |Y|) & \to &
\prod_{i\in I} \mathrm{Mor}_{\mathcal{SV}}(|X_i|,
|Y|) \begin{array}{c} \to \\
\to \end{array} & \prod_{i, i'\in I}\mathrm{Mor}_{\mathcal{
SV}}(|X_i|\cap |X_{i'}|, |Y|)
\end{array}
\end{equation*} 
that shows that one has to consider the case $X=SSp \ R$ only.
Suppose that ${\bf f}, {\bf g}\in\mathrm{Mor}_{\mathcal{SF}}(SSp \ R, Y)$ satisfy $|{\bf f}|=|{\bf g}|$. 
Fix an open affine covering 
$\{Y_i\}_{i\in I}$ of $Y$. Since $(|{\bf f}|^e)^{-1}(|Y_i|)=(|{\bf g}|^e)^{-1}(|Y_i|)$, it infers ${\bf f}^{-1}(Y_i)={\bf g}^{-1}(Y_i)$.
In particular, one can assume that $Y=SSp \ S$. Thus Yoneda's Lemma and Lemma \ref{canobijection} imply ${\bf f}={\bf g}$.
Finally, consider $f\in\mathrm{Mor}_{\mathcal{SV}}(SSpec \ R, |Y|)$.
The above diagram applied to the open affine covering $\{(SSp \ R)_{Q_i}\}_{i\in I}, Q_i=(f^e)^{-1}(|Y_i|)$, reduces the general case to
$Y=SSp \ S$.
\end{proof}
The following theorem, which we call the Comparison Theorem, is a superization of \emph{th$\Acute{e}$or$\Grave{e}$me de comparaison}, \cite{dg}, I, \S 1, 4.4. 
\begin{theorem}\label{comparison}
The functors $X\mapsto |X|$ and $Y\mapsto Y^{\diamond}$ define equivalences of the categories
$\mathcal{SF}$ and $\mathcal{SV}$ which are quasi-inverses of each other. 
\end{theorem}
\begin{proof}
Combine Lemma \ref{equivofcategories} and Proposition \ref{adjoint} with Lemmas \ref{fullfaithfunc} and \ref{fullfaithspace}.
\end{proof}

\begin{prop}\label{superschemeissheaf} Any superscheme is a
dur $K$-sheaf (and $K$-sheaf as well).
\end{prop}
\begin{proof} Let $X\in \mathcal{SF}$. By Comparison Theorem, there is $Y\in \mathcal{SV}$ such that
$X\simeq Y^{\diamond}$. It is enough to check that $Y^{\diamond}$ satisfies the following conditions
(cf. \cite{dg}, p.285, or \cite{z}, p.721). 

1. For a finite family of superalgebras $\{A_i\}_{i\in I}$ one has $Y^{\diamond}(\prod_{i\in I}A_i)\simeq\prod_{i\in I}Y^{\diamond}(A_i)$. 

2. If $B$ is an $A$-superalgebra that is a faithfully flat $A$-module, then the diagram
$$Y^{\diamond}(A)\to Y^{\diamond}(B) \begin{array}{c} \to \\
 \to\end{array} Y^{\diamond}(B\otimes_A B)$$   
is exact. 

Since $SSpec \ \prod_{i\in I} A_i$ is isomorphic to a direct sum of superschemes $SSpec \ A_i$, the first condition holds. The second condition holds whenever $SSpec \ B\to SSpec \ A$ is a surjective morphism and 
it is a cokernel of $SSpec \ B\otimes_A B \begin{array}{c} \to \\ \to\end{array} SSpec \ B$ (in the category $\mathcal{SV}$).
If $p\in   
(SSpec \ A)^e$, then $Bp\bigcap A=p$ (cf. \cite{bur}, I, \S 3, Proposition 9) and the multiplicative set $S=A\setminus p$ does not meet 
$Bp$. There is a $q\in (SSpec \ B)^e$ such that $q\bigcap S=\emptyset$ and therefore, $q\bigcap A=p$. In other words,
$(SSpec \ \iota^B_A)^e$ is surjective. Since for any $p\in |SSpec \ A|$ the morphism of stalks $(SSpec \ \iota^B_A)_p$
is injective, $SSpec \ \iota^B_A$ is surjective. Finally, $SSp \ B\otimes_A B\simeq SSp \ B\times_{SSp \ A} SSp \ B$ in the category $\mathcal{F}$. Comparison Theorem infers that $SSpec \ B\otimes_A B\simeq SSpec \ B\times_{SSpec \ A} SSpec \ B$ in the category $\mathcal{SV}$.  
\end{proof}

\section{Supergrassmannian}\label{Grassmannian}

Let $V$ be a supervector space of {\it superdimension} $m|n$, that is $\dim V_0=m, \dim V_1=n$.
Denote $m|n$ by $s\dim V$.
 
A general linear supergroup $GL(V)$, that is denoted also by $GL(m|n)$,
is a group $K$-functor such
that for any $A\in \mathsf{SAlg}_K$ the group $GL(V)(A)$ consists of all
even and $A$-linear automorphisms of $V\otimes A$.
The $A$-supermodule $V\otimes A$ is a {\it free $A$-supermodule}
$A^{m|n}$ of {\it superrank} $m|n$. The group
$GL(V)(A)$ acts freely and transitively on the bases of $A^{m|n}$ as a free supermodule.

A projective object in the category of (left or right) $A$-supermodules is called a {\it projective $A$-supermodule}.
By \cite{amas}, Lemma 5.1, an $A$-supermodule $P$ is projective iff it is projective as an $A$-module.

If $A$ is a local superalgebra, and if $P$ is a finitely generated projective $A$-supermodule, then
$P$ is free. Indeed, let $\mathfrak{m}$ be the unique maximal superideal of $A$. Notice that 
$\mathfrak{m}= rad \, A$ is the Jacobson radical
of $A$; see the second paragraph of Section 4. By choosing homogeneous elements
in $P$ which project onto homogeneous basis elements in the supervector space $P/\mathfrak{m}P$    
over the field $A/\mathfrak{m}$, 
we define such a morphism $A^{s|t}\to P$ of supermodules that induces an isomorphism
$A^{s|t}/\mathfrak{m}A^{s|t}\to P/\mathfrak{m}P$, where $s|t =  s\dim \, P/\mathfrak{m}P$.
Nakayama's Lemma proves that this morphism 
is an epimorphism, and is indeed an isomorphism since it splits by the projectivity of $P$; see
\cite{kash}, Theorem 9.2.1(d).
 
We say that a projective $A$-supermodule $P$ has a superrank $m|n$, 
whenever $P$ is finitely generated and for any $p\in (SSpec \ A)^e$ its localization
$P_p$ is a free $A_p$-supermodule of superrank $m|n$.

Define $nil\, A = \sqrt{0} \oplus A_1$, the nilradical of $A$; this equals the intersection $\bigcap\, p$ of all $p \in (SSpec\, A)^e$.
\begin{lemma}\label{ranklm}
A finitely generated $A$-supermodule $P$ is projective of superrank $m|n$ iff $P/(nil \, A)P$ is a projective $A/nil \, A$-supermodule of
superrank $m|n$.
\end{lemma}

\begin{proof}
We have a canonical isomorphism $P_p/(nil \, A)_p P_p\simeq (P/(nil \, A)P)_{\overline{p}}$, where $\overline{p}=p/nil \, A, p\in (SSpec \ A)^e$. 
Since $(nil \, A)_p\subseteq nil \, A_p \subseteq pA_p$, the same arguments as above conclude the proof.
\end{proof}

\begin{prop}\label{rankprop}
 (see also \cite{cf}, Appendix)
The following statements for a finitely generated $A$-supermodule $P$ are equivalent:
\begin{enumerate}
\item $P$ is projective of superrank $m|n$;
\item $P_p$ is a free $A_p$-supermodule of superrank $m|n$ 
for any maximal superideal $p$ of $A$;
\item For any maximal superideal $p$ of $A$, there is $f\in A_0\setminus p_0$ such that 
$P_f$ is a free $A_f$-supermodule of rank $m|n$;
\item There are elements $f_1, \ldots , f_t\in A_0$ such that $\sum_{1\leq i\leq t} A_0f_i=A_0$, and for any $i$, $P_{f_i}$ is a free
$A_{f_i}$-supermodule of superrank $m|n$. 
\end{enumerate} 
\end{prop}

\begin{proof}
The statements 3 and 4 are equivalent to the statement that the set
$$E=\{f\in A_0 | P_f \ \mbox{is a free} \ A_f -\mbox{supermodule of superrank} \ m|n\}$$ generates $A_0$ as an ideal, or equivalently, $E$ is not contained in any maximal superideal of $A$ (cf. \cite{z}, Lemma 1.1). 
Lemma \ref{ranklm}, the isomorphism from this lemma and \cite{bur}, II, \S 5, Theorem 2 as well, infer that
the statements 1 and 2 are equivalent to each other.
Finally, we have a canonical isomorphism  
$P_f/(nil \, A)_f P_f\simeq (P/(nil \, A) P)_{\overline{f}}$, where $\overline{f}$ is a residue class of $f$ in $A/nil \, A, f\in A_0$.
Observe that $(nil \, A)_f\subseteq nil \, A_f$. Again by \cite{bur}, II, \S 5, Theorem 2, the statements 2 and 3 are equivalent to each other.
\end{proof}

Fix non-negative integers $s, r$ such that $s\leq m, r\leq n$.
Define a $K$-functor 
$$Gr(s|r, m|n)(A)=\{M | M \ \mbox{is a direct summand of} \ A^{m|n} \ \mbox{of superrank} \ s|r\},$$
where $A \in \mathsf{SAlg}_K$. 
If $M\in Gr(s|r, m|n)(A)$, then for any superalgebra morphism $\phi : A\to B$ define $$Gr(s|r, m|n)(\phi)(M)=M\otimes_A B\subseteq A^{m|n}\otimes_A B=
B^{m|n}.$$
Since $(M\otimes_A B)_{p}\simeq M_q\otimes_{A_q} B_p, q=\phi^{-1}(p), p\in (SSpec B)^e$ (see Lemma 1.4, \cite{z1}), $M\otimes_A B$ has superrank $s|r$.      
To simplify our notations we denote $Gr(s|r, m|n)$ by $Gr$, if it does not lead to confusion.
\begin{lemma}\label{extensbylocalmorph}
Let $A\to B$ be a local morphism of local superalgebras, and let $M$ be a finitely generated (right) $A$-module. 
If $M\neq 0$, then $M\otimes_A B\neq 0$.
\end{lemma}

\begin{proof}
The induced morphism $\overline{A}\to \overline{B}$ is also local. Besides, we have an epimorphism $M\otimes_A B\to M/M(nil \, A)\otimes_{\overline{A}}
\overline{B}$. By Nakayama's Lemma $M/M(nil\, A)\neq 0$ and our statement follows by \cite{bur}, II, \S 4, Lemma 4. 
\end{proof}

\begin{lemma}\label{annihilator}
Let $M$ be a finitely generated $R$-supermodule, and let $A$ be an $R$-superalgebra. Then $M\otimes_R A=0$ iff
$\iota^A_R(I)A=A$, where $I=Ann_R(M)$.  
\end{lemma}

\begin{proof}
The part $``$if" is obvious. By Lemma 1.2 (i), \cite{z}, $R_p$ is a flat $R$-module for any $p\in (SSpec R)^e$. Therefore, Proposition 16 and 17 from \cite{bur}, II, \S 4, can be superized per verbatim. We have
$$Supp(M)=\{p\in (SSpec \ R)^e|M_p\neq  0\}=\{p\in (SSpec \ R)^e| I\subseteq p\}.$$  
Assume that $\iota^A_R(I)A\neq A$ and $q$ is a maximal superideal of $A$ that contains $\iota^A_R(I)$.
By Lemma 1.4, \cite{z1}, $M\otimes_R A=0$ implies $M_p\otimes_{R_p} \iota^A_R(R_0\setminus p_0)^{-1}A=0$, and moreover, $M_p\otimes_{R_p} A_q=0$, where $p=(\iota^A_R)^{-1}(q)\in (SSpec \ R)^e$. 
Lemma \ref{extensbylocalmorph} implies $M_p=0$, a contradiction.
\end{proof}

Let $R\in \mathsf{SAlg}_K$. Recall from the remark just above Proposition \ref{Noetherian} that we have 
the canonical identification ${}_R\mathsf{SMod} = \mathsf{SMod}_R$, and every object in these categories naturally turns into an 
$(R, R)$-superbimodule. Let $B$ be an $R$-superalgebra. We see that if $M \in {}_R\mathsf{SMod}$, then 
$b \otimes m \mapsto (-1)^{|b||m|} m \otimes b$ gives an isomorphism 
\begin{equation}\label{side-switch}
B \otimes_R M \simeq M \otimes_R B
\end{equation} 
in ${}_B\mathsf{SMod} = \mathsf{SMod}_B$.

\begin{lemma}\label{extenofhoms}
Let $R$, $B$ be as above, and assume that $B$ is flat as an $R$-module.
Let $M, N \in {}_R\mathsf{SMod}$, and assume that $M$ is finitely presented as a right $R$-module.

1. There is a natural $K$-linear morphism,  
$$\mathrm{Hom}_{{}_R\mathsf{Mod}}(M, N) \otimes_R B \overset{\simeq}{\longrightarrow} \mathrm{Hom}_{\mathsf{Mod}_B}(M\otimes_R B, N\otimes_R B).$$ 

2. Assume $M \subset N$, and that $B$ is a faithfully flat $R$-module. If $M \otimes_R B$ is a direct summand of $N\otimes_R B$
in $\mathsf{SMod}_B$, then $M$ is a direct summand of $N$ in ${}_R\mathsf{SMod}$.
\end{lemma}

\begin{proof}
1. The desired isomorphism is obtained as the composite of isomorphisms,
\begin{align*}
&\mathrm{Hom}_{{}_R \mathsf{Mod}}(M, N) \otimes_R B \simeq \mathrm{Hom}_{{}_R\mathsf{Mod}}(M, N \otimes_R B)\\ 
&\simeq \mathrm{Hom}_{{}_B\mathsf{Mod}}(B \otimes_R M, N \otimes_R B) \simeq \mathrm{Hom}_{\mathsf{Mod}_B}(M\otimes_R B, N\otimes_R B).
\end{align*}
The assumptions are used only for the first canonical isomorphism, for which $N$ is regarded merely as an 
$(R, R)$-bimodule. The last isomorphism is induced from the isomorphism \eqref{side-switch}. 

2. Present $\mathbb{Z}_2 = \langle g \mid g^2 =1\rangle$ as a multiplicative group generated by $g$. 
Recall that every supervector space $V$ is identified with the module over the group algebra $K\mathbb{Z}_2$
in which $g$ acts on homogeneous elements $v \in V$ by 
\begin{equation}\label{Z2action}
v \mapsto (-1)^{|v|}v.  
\end{equation}

In the general situation without assuming $M \subset N$, 
let the group $\mathbb{Z}_2$ act on $\mathrm{Hom}_{{}_R \mathsf{Mod}}(M, N)$ by conjugation; 
explicitly, the 
$g$-conjugation ${}^g\varphi$ of $\varphi$ is defined by ${}^g\varphi(m) = g\varphi(gm)$. 
Then the $\mathrm{Hom}$ space
turns into a supervector space with respect to the parity which corresponds, as recalled above, 
to the $\mathbb{Z}_2$ action just defined. 
Given an $R$-linear map $\varphi : M \to N$, its even part is given by 
$\varphi_0 = (1/2)(\varphi + {}^g\varphi)$. It follows that $\varphi_0$ is in ${}_R\mathsf{SMod}$, and 
$(\varphi \, \psi)_0 =  \varphi_0 \, \psi$ for any $\psi : L \to M$ in ${}_R\mathsf{SMod}$.

Let us be in the assumed situation. It follows from Part 1 above that the restriction morphism 
$\mathrm{Hom}_{{}_R \mathsf{Mod}}(N, M) \to \mathrm{Hom}_{{}_R \mathsf{Mod}}(M, M)$ is surjective, since it is so
with $\otimes_R B$ applied. Therefore we have an $R$-linear retraction $\phi : N \to M$. 
The argument in the preceding paragraph shows that $\phi_0$ is a retraction in ${}_R\mathsf{SMod}$.
\end{proof}

\begin{lemma}\label{gross is local}
$Gr$ is a local $K$-functor.
\end{lemma}
\begin{proof}
Let $R\in\mathsf{SAlg}_K$, and suppose that $f_1, \ldots, f_t\in R_0$ satisfy $\sum_{1\leq i\leq t}R_0 f_i=R_0$.
Fix a collection $\{M_i\in Gr(R_{f_i})\}_{1\leq i\leq t}$ such that for any $i\neq j$, the canonical images of $(M_i)_{f_j}$ and
$(M_j)_{f_i}$ in $R_{f_i f_j}^{m|n}$ coincide. For the canonical morphism of $R$-supermodules 
$R^{m|n}\to (R^{m|n})_{f_i}=R_{f_i}^{m|n}$, let $N_i$ denote the pre-image of $M_i$. It is clear that $M_i=(N_i)_{f_i}$. Set $N=\bigcap_{1\leq i\leq t} N_i$.
Take an element $n\in N_i$. There is an integral non-negative number $k$ such that $(f_if_j)^k n\in N_j$ for any $j$, hence $f_i^k n\in N_j$ for any $j$. Therefore, $N_{f_i}=(N_i)_{f_i} = M_i$ for any $i$.
It remains to prove that $N$ is a direct summand of $R^{m|n}$. 

Set $B=\prod_{1\leq i\leq t}R_{f_i}$. By Lemma 1.2 (ii), \cite{z}, $B$ is a fppf covering of $R$. Note that
$B \otimes_R N \simeq N\otimes_R B$ is a direct summand of
$B^{m|n}$, whence it is finitely presented, as a left and right $B$-module. 
By \cite{bur}, I, \S 3, Proposition 11, $N$ is a finitely presented, as a right, say, $R$-module.
It follows by Lemma \ref{extenofhoms} that $N$ is a direct summand of $R^{m|n}$, as desired.
\end{proof}

Let $W$ be a supersubspace of $V$ with $s\dim W =(m-s)|(n-r)$. Define a subfunctor $Gr_W$ of $Gr$ by $$ Gr_W(A)=\{M | M\bigoplus (W\otimes A)=A^{m|n}\},\ A \in \mathsf{SAlg}_K. $$

\begin{lemma}\label{Grlm}
$Gr_W$ is an open affine subfunctor of $Gr$.
\end{lemma}

\begin{proof}
Choose ${\bf g}_M : SSp \ R\to Gr$ defined by an element $M\in Gr(R)$. Then $\phi\in SSp \ R(A)$ belongs to
${\bf g}_M(A)^{-1}(Gr_W(A))$ iff the induced supermodule morphism $\mu_A : M\otimes_R A\to A^{m|n}/ W\otimes_R A$ is an isomorphism, where 
$A$ is an $R$-supermodule via $\phi$.   
Since the superranks of $M\otimes_R A$ and $A^{m|n}/ W\otimes_R A$ are the same, $\mu_A$ is an isomorphism iff it is an epimorphism. Both statements easily follow by Proposition \ref{rankprop} (2) combined with \cite{z}, 
Lemma 1.5. By Lemma \ref{annihilator}, ${\bf g}_M^{-1}(Gr_W)= D(I),$ where $I=Ann_R(R^{m|n}/(M + W\otimes R))$. The same arguments as in \cite{dg}, I, \S 1, 3.9, show that $Gr_W$ is isomorphic to the affine superscheme ${\bf A}^{u|v}$ (cf. \cite{z}, p.719), where $u=s(m-s)+r(n-r), v=s(n-r)+r(m-s)$.
\end{proof}

\begin{corollary}\label{Grcor}
The $K$-functor $Gr(s|r, m|n)$ is a superscheme.
\end{corollary}

Let $U$ be a supersubspace of $V$ such that $s\dim U=s|r$. Denote the stabilizer $Stab_{GL(V)}(U)$ (cf. \cite{z}, p.720) by
$P(U)$; this is a closed supersubgroup of $GL(V)$. 

\begin{prop}\label{Grprop}
The quotient $GL(V) \tilde{/} P(U)$ is isomorphic to $Gr(s|r, m|n)$.
\end{prop}

\begin{proof}
We have an embedding $(GL(V)/P(U))_{(n)}\to Gr$ given by $g \mapsto g(U\otimes A),\ g\in GL(V)(A),\ A\in \mathsf{SAlg}_K$.
Moreover, an element $M\in Gr(A)$ belongs to $(GL(V)/P(U))_{(n)}(A)$ iff $A^{m|n}=M\bigoplus P$ and both $M$ and $P$ are free $A$-supermodules of superranks $s|r$ and $(m-s)|(n-r)$, respectively. 
It remains to prove that $(GL(V)/P(U))_{(n)}$ is {\it dense} in $Gr$ with respect to the Grothendieck topology of fppf coverings. 

By Proposition \ref{rankprop} (4) for an element $M\in Gr(A)$ and its complement $P$ there are elements $f_1,\ldots , f_t, g_1, \ldots , g_{l}\in A_0$ such that $\sum_{1\leq i\leq t}A_0 f_i=\sum_{1\leq j\leq l}A_0 g_j=A_0$ and each $A_{f_i}$-supermodule $M_{f_i}$ (respectively,
each $A_{g_j}$-supermodule $P_{g_j}$) is free. Observe that Proposition 7 from \cite{bur}, II, \S 2, holds for any (not necessary commutative)
ring $A$ and any multiplicative subsets $S, T$ of its center. In particular, $A_{f_i g_j}$-supermodules $M_{f_i g_j}$ and $P_{f_i g_j}$ are free,
$1\leq i\leq t, 1\leq j\leq l$. Since $\sum_{1\leq i\leq t, 1\leq j\leq l}A_0f_i g_j=A_0$, Lemma 1.2 (ii), \cite{z}, concludes the proof.
\end{proof}

\section{Affiness criteria for quotient dur $K$-sheaves}\label{affinessdursheaf}

Let $G = SSp~D$ be an affine supergroup. Let $X = SSp~B$ be an affine superscheme on which 
$G$ acts from the right. Thus, $D$ is a Hopf superalgebra so that $\mathsf{SMod}^D$ forms a tensor category
in the obvious way, and $B$ is a right $D$-comodule superalgebra, that is, an algebra object in  
$\mathsf{SMod}^D$. Let $\rho : B \to B \otimes D$ denote the structure on $D$; it then gives rise to the
action $SSp~\rho : X \times G \to X$. Set $C = B^G$, the superalgebra of $G$-invariants in $B$. This coincides with
\begin{equation}\label{coinv}
B^{coD} = \{b \in B \mid \rho(b) = b \otimes 1 \}.
\end{equation}
Recall from Section \ref{K-functors} that $X \tilde{/} G$ (resp.,  $X \tilde{\tilde{/}} G$) denotes the sheafification
(resp., the dur sheafification) of the naive quotient $(X/G)_{(n)}$ of $X$ over $G$. 

We let $\mathsf{SMod}^D_B$ denote the
category $(\mathsf{SMod}^D)_B$ of right $B$-module objects in $\mathsf{SMod}^D$, and call an object in this category
a $(D, B)$-Hopf supermodule; this notion is a superization of the relative Hopf modules which were defined by Takeuchi \cite{tak}.
Given an object $N$ in $\mathsf{SMod}_C$, the right $B$-supermodule $N \otimes_C B$, given $\mathrm{id}_N \otimes \rho :
N \otimes_C B \to (N \otimes_C B) \otimes D$, turns into an object in $\mathsf{SMod}_B^D$, so that we have 
the functor
\begin{equation}\label{Hopfmodfunctor}
\mathsf{SMod}_C \to \mathsf{SMod}_B^D, \quad  N \mapsto N \otimes_C B.
\end{equation}
This will be seen to be a tensor functor; see Remark \ref{tensorfunctor} below.

The next theorem will play a crucial role when we prove Theorem \ref{mainthm} and Proposition 
\ref{ev of factor is factor of ev}
in the following two sections. This theorem generalizes Satz A of Oberst \cite{ober} to the
super situation; see also Schneider \cite{sch}, Theorem I, which proves Oberst's Satz A in
the non-commutative situation.

\begin{theorem}\label{superOberst}
Let the notation be as above.

1. The following are equivalent: 
\begin{itemize}
\item[(1)]
$G$ acts on $X$ freely, and the dur $K$-sheaf $X \tilde{\tilde{/}} G$ is affine;
\item[(2)]
\begin{itemize}
\item[(a)]
$B$ is injective as a right $D$-comodule, and
\item[(b)]
the map $$\alpha = \alpha_B : B \otimes B \to B \otimes D,\ \alpha(b \otimes b') = b \rho (b')$$ 
is a surjection;
\end{itemize}
\item[(3)]
\begin{itemize}
\item[(a)]
$B$ is faithfully flat over $C$, and
\item[(b)]
the map $$\beta = \beta_B : B \otimes_C B \to B \otimes D,\ \beta(b \otimes_C b') = b \rho (b')$$ 
induced from the $\alpha_B$ above is a bijection.
\end{itemize}
\item[(4)]
The functor $\mathsf{SMod}_C \to \mathsf{SMod}_B^D$ given in \eqref{Hopfmodfunctor} is a category equivalence. 
\end{itemize}

If these conditions are satisfied, then $X \tilde{\tilde{/}} G = SSp~C$. 

2. Suppose that $G$ is algebraic, or in other words, $D$ is finitely generated. Suppose that $B$ is Noetherian;
see Proposition \ref{Noetherian}. 
If the equivalent conditions above are satisfied, 
then $X \tilde{\tilde{/}} G = SSp~C$ is Noetherian (or equivalently, $C$ is Noetherian), and it coincides with 
the $K$-sheaf $X \tilde{/} G$.
\end{theorem}

\begin{proof}
The equivalence (1) $\Leftrightarrow$ (3) as well as the last statement of Part 1 has been proved by
the second named author \cite{z}, Section 4; see Proposition 4.2 for (3) $\Rightarrow$ (1), and Proposition 4.1
and the second sentence of its proof for the rest.  
If the condition (3) is satisfied, the assumptions of Part 2 imply that $C$ is Noetherian, and $B \geq C$.
Therefore, Part 2 follows again by
\cite{z}, Proposition 4.2. 

The equivalence (3) $\Leftrightarrow$ (4) is proved in the same way of proving
\cite{sch}, Theorem 3.7, in its special case when $H = \overline{H}$. In fact, (3) $\Rightarrow$ (4) can be
alternatively understood as the faithfully flat descent theorem, which proves under the assumption (3)(a), 
the equivalence between $\mathsf{SMod}_C$ and the category of right comodules over the natural coring $B \otimes_C B$
in $\mathsf{SMod}_K$, since the latter category is identified with $\mathsf{SMod}_B^D$ through the isomorphism
assumed by (3)(b). On the other hand, the functor given in \eqref{Hopfmodfunctor} has the right
adjoint $M \mapsto M^{coD}$ (see \eqref{coinv}), where $M \in \mathsf{SMod}_B^D$,
and the map $\beta_B$ is the adjunction for $B \otimes D$. This together with the faithful exactness 
of the equivalence shows (4) $\Rightarrow$ (3).

We postpone proving the remaining (2) $\Leftrightarrow$ (3) until Section \ref{bosonization}.
\end{proof}

\begin{rem}\label{tensorfunctor}
Keep the notation as above. Given an object $N$ in $\mathsf{SMod}_C$, regard it as a $(C, C)$-superbimodule,
by defining the left $C$-action as in \eqref{leftaction}. Then, $\mathsf{SMod}_C$ forms
a tensor category with respect to the tensor product $N \otimes_C N'$ and the unit $C$. Similarly, $\mathsf{SMod}_B^D$
forms a tensor category. We see that the functor given in \eqref{Hopfmodfunctor} is a tensor functor. 
\end{rem} 

\section{Proof of the main theorem}\label{proofmain}

This section is devoted mostly to prove our Main Theorem \ref{mainthm}.

\subsection{}\label{proof1}
Given $X \in \mathcal{F}$, we let $X_{res} = X|_{\mathsf{Alg}_K}$ denote the functor restricted to $\mathsf{Alg}_K$. Thus we have the functor
$$\mathcal{F} = \mathsf{Sets}^{\mathsf{SAlg}_K} \to \mathsf{Sets}^{\mathsf{Alg}_K},\ X \mapsto X_{res}.$$ 
The geometric realization of the functor $X_{res}$ (cf. \cite{dg}, I, \S 1, 4.2)
is also denoted by $|X_{res}|$.
Given $A \in \mathsf{SAlg}_K$, recall from \eqref{overlineA} the definition of $\overline{A}$. Let $\pi_A : A \to \overline{A}$
denote the quotient map.  
 
\begin{lemma}\label{coincidence}
Let $X \in \mathcal{F}$.

1. We have the coincidence  $|X|^e=|X_{res}|^e$ of topological spaces. 

2. $U \mapsto U_{res}$ gives a bijection, $Op(X) \overset{\simeq}{\longrightarrow} Op(X_{res})$.
\end{lemma}

\begin{proof} 
1. By Lemma \ref{natiso} 
we have coincidence $|X|^e=|X_{res}|^e$ of sets.
Again by Lemma \ref{natiso}, $V\subseteq X(B)$ is open iff for any pair $(R, x)\in \mathcal{M}_{X}$ there is a superideal
$I_x$ in $R$ such that ${\bf g}_x(B)^{-1}(V)=D(I_x)(B)$. Denote $X(\pi_R)(x)\in X(\overline{R})$ by
$\overline{x}$. It remains to notice that $SSp \ R(B)=SSp \ \overline{R}(B)$, ${\bf g}_x^{-1}(V)= {\bf g}^{-1}_{\overline{x}}(V)$ and
$D(I_x)(B)=D(\overline{I_x})(B),$ where $\overline{I_x}=\pi_R(I_x)$. 

2. This follows from Part 1 above and Proposition \ref{open-open}.
\end{proof}
 
\begin{rem}\label{varcoincidence}
The lemma above can be generalized so that if $Y\subseteq X$ is a subfunctor 
such that $X_{res} = Y_{res}$, then $|X|^e=|Y|^e$ and $Op(X) \overset{\simeq}{\longrightarrow} Op(Y)$.  
\end{rem}

\begin{lemma}\label{flat}
Suppose $B \in \mathsf{SAlg}_K$, $C \in \mathsf{Alg}_K$. Then, $B\geq C$ iff $B_0\geq C$ and $B_1$ is a flat $C$-module that is finitely
presented as a $B_0$-module. 
\end{lemma}

\begin{proof}
The superalgebra $B$ is a faithfully flat $C$-(super)module iff $B_0$ is a faithfully flat $C$-module and $B_1$ is a flat
$C$-module (cf. \cite{bur}, I, \S 2, Proposition 2). Assume that $B=C[k|l]/I$, where $C[k|l]$ is a {\it free commutative} $C$-superalgebra with {\it free even} generators $t_1, \ldots , t_k$ and {\it free odd} generators $z_1, \ldots , z_l$. Besides, $I$ is a finitely generated superideal of $C[k|l]$.
Let $f_1, \ldots f_d, f_{d+1}, \ldots, f_s$ be homogeneous generators of $I$ such that $|f_i|=0$ if $i\leq d$, 
and $|f_i|=1$ otherwise. Thus $I\bigcap C[k|l]_0$ is generated (as a $C[k|l]_0$-ideal) by $f_i, f_r z_j, 1\leq i\leq d, d+1\leq r\leq s, 1\leq j\leq l$.  
Since $C[k|l]_0$ is a finitely presented $C$-algebra, it infers that $B_0$ is. Analogously, $B_1=(\bigoplus_{1\leq j\leq l}C[k|l]_0z_j)/I_1$ and
$I_1$ is generated (as a $C[k|l]_0$-module) by $f_iz_j, f_r$. 
Conversely, if $B_1$ is a finitely presented $B_0$-module, then $B$ is a finitely presented $B_0$-superalgebra. 
\end{proof} 

Let $Y$ be a $K$-functor such that for any monomorphism of superalgebras $\phi : A\to B$, the map
$Y(\phi)$ is injective. The naive quotient $(G/H)_{(n)}$ gives an example of such a functor.
Assume additionally that $Y$ commutes with the direct products of superalgebras. In particular, $Y$ is suitable; see Definition \ref{def-suitable}.
It is clear that $Y_{ev}$ inherits both of the above properties of $Y$. In particular, $\widetilde{Y_{ev}}$ is a subfunctor of $\tilde{Y}$. Denote the 
sheafification of $Y_{res}$ in the category $\mathsf{Sets}^{\mathsf{Alg}_K}$ by $\widetilde{Y_{res}}$.

\begin{lemma}\label{identify}
If $Y$ is as above, then the functors $(\widetilde{Y_{ev}})_{res}$ and $\widetilde{Y_{res}}$ are canonically identified.
\end{lemma}

\begin{proof}
For any $B\in \mathsf{SAlg}_K$ we identify the sets $Y_{ev}(B)$ and $Y(B_0)$ via $Y(\iota^B_{B_0})$. 
If $C\in \mathsf{Alg}_K$ and $y\in \widetilde{Y_{ev}}(C)$, then there is $B\geq C$ such that
$y$ belongs to $\mathrm{Ker}(Y(B_0)\begin{array}{c}\to \\
\to\end{array} Y((B\otimes_C B)_0))$.
By Lemma \ref{flat} $B_0\geq C$. We have a commutative diagram 
$$\begin{array}{ccc}
Y(B_0) & \begin{array}{c}\to \\
\to\end{array} & Y((B\otimes_C B)_0) \\
\uparrow & & \uparrow \\
Y(B_0) & \begin{array}{c}\to \\
\to\end{array} & Y(B_0\otimes_C B_0)
\end{array}
$$
whose vertical arrows are embeddings (the left arrow is an identity map).
Thus $y$ belongs to $\mathrm{Ker}(Y(B_0)\begin{array}{c}\to \\
\to\end{array} Y(B_0\otimes_C B_0))$, that is $y$ represents an element $y'\in \widetilde{Y_{res}}(C)$.
The definition of $y'$ does not depend on $B$. If $z\in Y(D_0)$ also represents $y$, where $D\geq C$, then
$C\leq B_0, D_0\leq R=B_0\otimes_C D_0$ and 
$Y(\iota^R_{B_0})(y)=Y(\iota^R_{D_0})(z)$ belongs to 
$\mathrm{Ker}(Y(R)\begin{array}{c}\to \\
\to\end{array} Y(R\otimes_C R))$. 
The map $y\mapsto y'$ is obviously a bijection. We leave it to the reader to check its functoriality.
\end{proof}
     
\begin{corollary}\label{identifycor}
Let $G$ be an algebraic supergroup, and let $H \subseteq G$ be a closed supersubgroup. 

1. We have $(G_{ev} \tilde{/} H_{ev})_{res} = G_{res} \tilde{/} H_{res}$.

2. $U \mapsto U_{res}$ well defines a map $Op(G_{ev} \tilde{/} H_{ev}) \to Op(G_{res} \tilde{/} H_{res})$,
and the thus defined map is a bijection.
\end{corollary}

\begin{proof}
1. This follows by Lemma \ref{identify} applied to $(G/H)_{(n)}$, since
$((G/H)_{(n)})_{res} =(G_{res}/H_{res})_{(n)}$.

2. This follows by Part 1 above and Part 2 of Lemma \ref{coincidence}.
\end{proof}

\begin{rem}\label{GL/P}
Let $U \subseteq V$ be supervector spaces such as in Proposition \ref{Grprop}. Then,
$$(GL(V) \tilde{/} P(U))_{res} = 
GL(V_0) \tilde{/} P(U_0) \times GL(V_1) \tilde{/} P(U_1) = GL(V)_{res} \tilde{/} P(U)_{res}.$$
This follows from Proposition \ref{Grprop}, since we see that the supergrassmannian $Gr(s|r, m|n)$,
restricted to $\mathsf{Alg}_K$, is canonically isomorphic to the direct product $Gr(s, m) \times Gr(r, n)$
of Grassmannians. Indeed, if $R \in \mathsf{Alg}_K$, then an $R$-supermodule $M$ is a direct summand of $R^{m|n}=V \otimes R$ iff 
$M_i$ is a direct summand of $V_i \otimes R,\ i=0, 1$.
\end{rem}

In what follows 
let $G$ be an algebraic supergroup, and $H \subseteq G$ a closed supersubgroup. Since one sees
$(G / H)_{(n)} \supseteq (G_{ev} / H_{ev})_{(n)}$, it follows that
$$G \tilde{/} H \supseteq G_{ev} \tilde{/} H_{ev}.$$

\begin{prop}\label{surjection}
The map
\begin{equation}\label{mapOp1}
Op(G \tilde{/} H) \to Op(G_{ev} \tilde{/} H_{ev})
\end{equation}
given by $U \mapsto U \bigcap (G_{ev}/H_{ev})$ is a surjection.
\end{prop}

\begin{proof}
By \cite{z}, Proposition 6.3, there exists a faithful representation $G\to GL(V)$ such that 
$H = Stab_{G}(U)$ for a supersubspace $U$ of $V$.
With the range restricted to $\mathsf{Alg}_K$,
the group scheme $G_{res}$ acts on the scheme $$(GL(V) \tilde{/} P(U))_{res} = GL(V)_{res} \tilde{/} P(U)_{res}$$ by left multiplications, and $H_{res}$ coincides with the centralizer
of the element $e_K P(U)_{res}(K)$, where $e_K$ is a unit of $GL(V)_{res}(K)$. 
By Proposition 5.2 of \cite{dg}, III, \S 3, the canonical morphism
$G_{res} \tilde{/} H_{res} \to GL(V)_{res} \tilde{/} P(U)_{res}$ is an immersion.
It follows by Corollary \ref{identifycor} that the map
\begin{equation}\label{mapOp2}
Op(GL(V)_{ev} \tilde{/} P(U)_{ev}) \to Op(G_{ev} \tilde{/} H_{ev})
\end{equation}
induced from the canonical $G_{ev} \tilde{/} H_{ev} \to GL(V)_{ev} \tilde{/} P(U)_{ev}$
is a surjection. 
We see from Lemma \ref{coincidence}, Corollary \ref{identifycor} and Remark \ref{GL/P} that the map \eqref{mapOp1} for $GL(V)$,\ $P(U)$ 
\begin{equation*}\label{maponOp2}
Op(GL(V) \tilde{/} P(U)) \to Op(GL(V)_{ev} \tilde{/} P(U)_{ev})
\end{equation*}
is a bijection.
Since the composite of this last bijection with the surjection \eqref{mapOp2} factors through \eqref{mapOp1},
the desired surjectivity follows.  
\end{proof}

\begin{lemma}\label{G_ev/H_ev}
$G_{ev} \tilde{/} H_{ev}$ is a superscheme.
\end{lemma}

\begin{proof}
Let $U'' \in Op(G_{res} \tilde{/} H_{res})$ be affine. By 
Corollary \ref{identifycor}, Part 2,
there exists a unique $U' \in Op(G_{ev} \tilde{/} H_{ev})$ such that $(U')_{res} = U''$.
Denote the pre-images of $U'$, $U''$ by $V' \in Op(G_{ev})$, $V'' \in Op(G_{res})$, respectively.
Obviously, $(V')_{res} = V''$. In addition, $V'$ and $V''$ are stable under the (free) actions (from the right) by
$H_{ev}$ and $H_{res}$, respectively, and 
\begin{equation}\label{eq-quotient}
V' \tilde{/} H_{ev} = U', \quad V'' \tilde{/} H_{res} = U''.
\end{equation}
To see this for $U'$ and $V'$, for example, note that $V'$, being the pre-image of $U' \cap (G_{ev}/H_{ev})_{(n)}$,
is stable under the $H_{ev}$-action, and $(V'/H_{ev})_{(n)} = U' \cap (G_{ev}/H_{ev})_{(n)}$, whose sheafification
is obviously $U'$.  
 
By \cite{jan},
Part I, 5.7(1), $V''$ is affine. Suppose $V'' = Sp~B$, where $B \in \mathsf{Alg}_K$. 
Then, $U'' \simeq Sp~B^{H_{res}}$. We see that $H_{ev}$, $B$ and $B^{H_{ev}}(= B^{H_{res}})$ satisfy 
the condition (3) (for $G$, $B$ and $C$) of Theorem \ref{superOberst}, Part 1, and so
that $U'$ is affine, and is isomorphic to $SSp~B^{H_{ev}}$.

Recall that Theorem \ref{mainthm} was proved in the non-super situation by Demazure and
Gabriel \cite{dg}, III, \S 3, 5.4. This classical result ensures 
that $G_{res} \tilde{/} H_{res}$ has a finite open affine covering. If $U''$ ranges over 
such a covering, then the corresponding $U'$ form a finite open covering of $G_{ev} \tilde{/} H_{ev}$
which is affine, as was just seen.
\end{proof}

\subsection{}\label{proof2}
We need some purely Hopf-algebraic argument. Keep $G$, $H$ as above. 
Suppose $G = SSp~A$. Thus, $A$ is a finitely generated Hopf superalgebra. 
Let $\varepsilon : A \to K$ denote the counit of $A$. For every supersubalgebra, say $R$, of 
$A$, we suppose that it has $\varepsilon|_R : R \to K$
as counit, and let $R^+ = \mathrm{Ker}(\varepsilon|_R)$ denote its kernel.
The largest purely even quotient algebra 
$\overline{A} = A/AA_1$ of $A$
is now a quotient Hopf algebra.
We have
$$ G_{ev} = SSp~\overline{A}, \quad  G_{res} = Sp~\overline{A}.$$
Define
$$ W^A = A_1/A_0^+A_1;$$
this is the odd part of the cotangent space of $G$ at 1. 
Note that $A \mapsto \overline{A}$,
$A \mapsto W^A$ are functorial. Since $A$ is supposed to be finitely generated, $\overline{A}$
is finitely generated, and $W^A$ is finite-dimensional; see \cite{amas}, Proposition 4.4.
Regard $A$ as a left $\overline{A}$-comodule
superalgebra along the quotient map $\pi = \pi_A : A \to \overline{A}$. Let $\Delta : A \to 
A \otimes A,\ \Delta(a) =  \sum a_1 \otimes a_2$ denote the coproduct on $A$. Set
\begin{equation}\label{lcoinv}
R = {}^{co\overline{A}}A = \{ a \in A \mid \sum \pi(a_1) \otimes a_2 = \pi(1) \otimes a \}.
\end{equation}
This is a right coideal supersubalgebra of $A$, that is, a supersubalgebra such that $\Delta(R) \subset R \otimes A$.
Recall from \cite{amas} the following result.

\begin{prop}\label{prop-isom} (\cite{amas}, Theorem 4.5)\
There is a counit-preserving isomorphism 
\begin{equation}\label{isom}
A \overset{\simeq}{\longrightarrow} \overline{A} \otimes \wedge(W^A)
\end{equation}
of left $\overline{A}$-comodule superalgebras. It follows that there is a counit-preserving 
isomorphism $R \simeq \wedge(W^A)$ of superalgebras.
\end{prop}

Suppose $H = SSp~D$. Thus, $D$ is a quotient Hopf superalgebra of $A$. We have $\overline{D}$, $W^D$, as above.
Set $T = {}^{co\overline{D}}D$; see \eqref{lcoinv}. Then, $T$ has a special counit with kernel $T^+$, and
$T \simeq \wedge(W^D)$, as above. 
We can choose an isomorphism such as \eqref{isom} and a similar isomorphism 
$D \overset{\simeq}{\longrightarrow} \overline{D} \otimes \wedge(W^D)$ so that the following
diagram commutes; see \cite{amas}, Remark 4.8.
\begin{equation*}
\begin{CD}
A @>{\simeq}>> \overline{A} \otimes \wedge(W^A) \\
@V{f}VV @VV{\overline{f} \otimes \wedge(W^f)}V \\
D @>{\simeq}>> \overline{D} \otimes \wedge(W^D)
\end{CD}
\end{equation*}
It results that the quotient map $A \to D$ restricts to a surjection $R \to T$,
which is identified with $\wedge(W^f) : \wedge(W^A) \to \wedge(W^D)$. 
Let $\mathfrak{a} = \mathrm{Ker}(R \to T)$ denote the kernel. 

\begin{lemma}\label{RtoT}
This $R \to T$ is a counit-preserving surjection of right $D$-comodule superalgebras, and the kernel $\mathfrak{a}$ is
nilpotent. 
\end{lemma}

\begin{proof}
The first assertion is easy to see. The remaining nilpotency follows since one sees that via the identification
given above, $\mathfrak{a}$ is generated by $\mathrm{Ker}(W^f : W^A \to W^D)$, so that
$\mathfrak{a}^{d+1} = 0$ if $d = \mathrm{dim}_K\mathrm{Ker}(W^f)$.
\end{proof}

By Lemma \ref{G_ev/H_ev}, $G_{ev} \tilde{/} H_{ev}$ has a finite open affine covering, say $\{ U'_i \}_{1 \leq i \leq n}$. 
Proposition \ref{surjection} ensures that for each $1 \leq i \leq n$, there exists $U_i \in Op(G \tilde{/} H)$ such that
$U_i \cap G_{ev} \tilde{/} H_{ev} = U'_i$. 

\begin{lemma}\label{U_i}
$\{ U_i \}_{1 \leq i \leq n}$ is an open covering of $G \tilde{/} H$.
\end{lemma}

\begin{proof} 
It suffices to prove that if $L$ is an algebraically closed field in $\mathsf{F}_K$,
then $(G \tilde{/} H)(L) \subseteq \bigcup_i U_i(L)$. Note that $U_i(L) \supseteq U'_i(L)$.
One sees that if $X \in \mathcal{F}$ is suitable, then $ X(L) = \tilde{X} (L)$ for $L$ as above.
This applied to $(G/H)_{(n)}$, $(G_{ev}/H_{ev})_{(n)}$ implies that $(G \tilde{/} H)(L) = 
(G_{ev} \tilde{/} H_{ev})(L)$. The desired result follows since we have 
$(G_{ev} \tilde{/} H_{ev})(L) = \bigcup_i U'_i(L)$.
\end{proof}

Theorem \ref{mainthm} will follow from the last lemma if we prove that each $U_i$ is affine and Noetherian.

To remove the index $i$ for simplicity, we re-start by choosing $U \in Op(G \tilde{/} H)$, 
$U' \in Op(G_{ev} \tilde{/} H_{ev})$ so that $U'$ is affine, and 
$U \cap G_{ev} \tilde{/} H_{ev} = U'$. Our aim is to prove that $U$ is affine and Noetherian.
Let $V$, $V'$ be the pre-images of $U$, $U'$, respectively. 
As was seen in the proof of Lemma \ref{G_ev/H_ev}, $V'$ is affine.
Suppose $V' = SSp~\overline{B}$ with $\overline{B} \in \mathsf{Alg}_K$, keeping an algebra morphism 
$\overline{A} \to \overline{B}$ in mind. Identify $A = \overline{A} \otimes R$ via
a fixed isomorphism such as \eqref{isom}. Set $B = \overline{B} \otimes R$; note that
the $\overline{B}$ given above coincides with $B/BB_1$, as is expected from the notation. Since
$V$ is the unique element in $Op(G)$ whose intersection with $G_{ev}$ equals $V'$, we see
$V = SSp~B$. Indeed, this $SSp~B$ is open in $G$ by 
Lemma \ref{openaffine}, 
and its intersection with $G_{ev}$ equals $SSp~\overline{B}$, as is seen from a push-out diagram 
in $\mathsf{SAlg}_K$. Since $\overline{B}$ is Noetherian,
it follows that $B$ is Noetherian; see Proposition \ref{Noetherian}.  

The same argument as proving \eqref{eq-quotient} shows that
$V = SSp~B$ is stable under the action by $H = SSp~D$ from the right,
and $V \tilde{/} H = U$. This $H$-action on $V$ is obviously free, and makes $B$ into a right $D$-comodule
superalgebra. 
By applying Theorem \ref{superOberst} on $G, X$ to
$H, V$, we see that the aim of ours above is equivalent to proving the next proposition.

\begin{prop}\label{prop-aim}
The right $D$-comodule superalgebra $B$ satisfies the following conditions given in (2) of Theorem \ref{superOberst},
Part 1:
\begin{itemize}
\item[(a)]
$B$ is injective as a right $D$-comodule, and
\item[(b)]
the map $$\alpha = \alpha_B : B \otimes B \to B \otimes D,\ \alpha(b \otimes b') = b \rho (b')$$ 
is a surjection, where $\rho : B \to B \otimes D$ denotes the structure on $B$.
\end{itemize}
\end{prop}

\subsection{}\label{proof3}
Recall $\mathfrak{a} = \mathrm{Ker}(R \to T)$ from Lemma \ref{RtoT}.
Set
\begin{equation}\label{mathfrakB}
\mathfrak{B} = B/B\mathfrak{a}.
\end{equation}
We list up the properties of $B$, $\overline{B}$ and $\mathfrak{B}$ which will be needed.
\begin{itemize}
\item[(i)]
$B$ is a right $D$-comodule superalgebra which includes $R$ as a right $D$-comodule supersubalgebra.
\item[(ii)]
$B/BR^{+} = \overline{B}$; this is naturally a right $\overline{D}$-comodule algebra.
\item[(iii)]
$\mathfrak{B}$ is a right $D$-comodule superalgebra which includes $T$ as a right $D$-comodule supersubalgebra.
\item[(iv)]
$\mathfrak{B}/\mathfrak{B}T^{+} = \overline{B}$.
\end{itemize}

\begin{lemma}\label{lem-overlineB}
The right $\overline{D}$-comodule algebra $\overline{B}$ satisfies the conditions
\begin{itemize}
\item[(c)]
$\overline{B}$ is injective as a right $\overline{D}$-comodule, and
\item[(d)]
the map $$\alpha_{\overline{B}} : \overline{B} \otimes \overline{B} \to \overline{B} \otimes \overline{D},\ 
\alpha_{\overline{B}}(b \otimes b') = b \overline{\rho}(b')$$ 
is a surjection, where $\overline{\rho}: \overline{B} \to \overline{B} \otimes \overline{D}$ denotes the structure on $\overline{B}$.
\end{itemize}
\end{lemma}

\begin{proof} 
This follows from Theorem \ref{superOberst} (or rather from the original, non-super version,
Satz A of \cite{ober}), since $\overline{B}$ satisfies the conditions given in (3) of the theorem,
as was seen in the proof of Lemma \ref{G_ev/H_ev}.
\end{proof}

\begin{lemma}\label{injectivity}
Every object in $\mathsf{SMod}_{\mathfrak{B}}^D$ 
is injective as a right $D$-comodule. 
\end{lemma}

\begin{proof}
By (c) of the last lemma, Doi's theorem, Theorem 1 of \cite{doi}, ensures that 
every $( \overline{D}, \overline{B})$-Hopf (non-super) module is 
$\overline{D}$-injective.
Let $ M \in \mathsf{SMod}_{\mathfrak{B}}^D$. By (iii), (iv), we see $M/MT^{+} \in \mathsf{SMod}_{\overline{B}}^{\overline{D}}$. This last object is 
especially a
$(\overline{D}, \overline{B})$-Hopf module, which is $\overline{D}$-injective, as was just seen.
We see from 
\cite{amas}, Proposition 1.1
that if $ N \in \mathsf{SMod}_T^D$, then $N/NT^+ \in \mathsf{SMod}^{\overline{D}}$, and that the assignment $ N \mapsto N/NT^+$ gives 
a category equivalence 
\begin{equation}\label{equiv}
\mathsf{SMod}_T^D \approx \mathsf{SMod}^{\overline{D}}. 
\end{equation}
Note that $M \in \mathsf{SMod}_T^D$, to which assigned is $M/MT^{+}$.  
Since the last object is injective in $\mathsf{SMod}^{\overline{D}}$ by Proposition \ref{injcomod}, $M$ is 
injective in $\mathsf{SMod}_T^D$, whence the structure morphism $M \to M \otimes D$ splits in $\mathsf{SMod}_T^D$. This implies
that $M$ is $D$-injective.
\end{proof}

Now, we can complete the proof of Theorem \ref{mainthm} by proving Proposition \ref{prop-aim}.

\vspace{\baselineskip}
\noindent
\emph{Proof of Proposition \ref{prop-aim}.} 
Recall from Lemma \ref{RtoT} that $\mathfrak{a}$ is nilpotent. Let $r > 0$ be an integer such that $\mathfrak{a}^r = 0$. 
By the last lemma, every object in $\mathsf{SMod}_{\mathfrak{B}}^D$ is $D$-injective. For each $0 \leq i < r$,
we have $B\mathfrak{a}^i/B\mathfrak{a}^{i+1} \in \mathsf{SMod}_{\mathfrak{B}}^D$; this is therefore $D$-injective. 
It follows that the short exact sequence 
\begin{equation*}
0 \to B\mathfrak{a}^i/B\mathfrak{a}^{i+1} \to B/B\mathfrak{a}^{i+1} \to B/B\mathfrak{a}^{i} \to 0
\end{equation*}
in $\mathsf{SMod}_B^D$ splits $D$-colinearly. We see that $B$ decomposes as 
\begin{equation*}
B \simeq \bigoplus_{i=0}^{r-1} B\mathfrak{a}^i/B\mathfrak{a}^{i+1}
\end{equation*}
into $D$-injective direct summands, whence it is $D$-injective. We have thus verified the condition (a).

To verify (b), it suffices to prove that 
the base extension $\overline{B} \otimes_{B} \alpha_B : \overline{B} \otimes B \to \overline{B} \otimes D$
of $\alpha_B$ along $B \to \overline{B}$ is a surjection, since the kernel $\mathrm{Ker}(B \to \overline{B})$ is nilpotent. Note that 
this $\overline{B} \otimes_{B} \alpha_B$ factors through
\begin{equation}\label{factor}
\overline{B} \otimes \mathfrak{B} \to \overline{B} \otimes D,\ b \otimes b' \mapsto b \rho'(b'),
\end{equation}
where $\rho' : \mathfrak{B} \to \mathfrak{B} \otimes D$ denotes the structure on $\mathfrak{B}$. 
It suffices to prove that this \eqref{factor} is a surjection.
This is a morphism in $\mathsf{SMod}_T^D$, which corresponds 
via the category equivalence \eqref{equiv} to the surjection $\alpha_{\overline{B}}$ in (d) of Lemma \ref{lem-overlineB}.
Therefore, the map \eqref{factor} and hence $\alpha_B$ are surjections, as desired. \quad $\Box$

\subsection{}\label{subsec-cor}

Here is a corollary to the proof given in the preceding subsection. 

\begin{corollary}\label{cor1} 
Let $G$ be an algebraic supergroup, and let $H$ be a 
closed supersubgroup of $G$. Then the following are equivalent:
\begin{itemize}
\item[(1)]
$G \tilde{/} H$ is affine;
\item[(2)]
$G_{ev} \tilde{/} H_{ev}$ is affine;
\item[(3)]
$G_{res} \tilde{/} H_{res}$ is affine.
\end{itemize}
Moreover, these are equivalent to the analogous conditions for $H \tilde{\backslash} G$, 
$H_{ev} \tilde{\backslash} G_{ev}$ and $H_{res} \tilde{\backslash} G_{res}$.
\end{corollary}

\begin{proof}
Suppose $G = SSp~A,\ H = SSp~D$, as above. From Theorem \ref{superOberst} 
and the analogous statement for $H \tilde{\backslash} G$, we see 
that the condition (1) above and the analogous one for $H \tilde{\backslash} G$ 
are both equivalent to
\begin{itemize}
\item[(1$'$)]
$A$ is injective as a right or equivalently left $D$-comodule,
\end{itemize}
and that the conditions (2), (3) and the analogous two conditions are all equivalent to 
\begin{itemize}
\item[(2$'$)]
$\overline{A}$ is injective as a right or equivalently left $\overline{D}$-comodule.
\end{itemize}
As is seen from isomorphisms such as \eqref{isom},\ $A$ (resp., $D$) is cofree as a left (and right) comodule over 
$\overline{A}$ (resp., $\overline{D}$). 
This shows that (1$'$) $\Rightarrow$ (2$'$), while the proof of
Proposition \ref{prop-aim} shows (1$'$) $\Leftarrow$ (2$'$).
\end{proof}

\section{Some properties of quotients}

\begin{lemma} \label{Y=X}
Let $X$ be a local $K$-functor in $\mathcal{F}$, and let $Y \subseteq X$ be a local subfunctor. Suppose that $\{ X_i \}_{i \in I}$
is an open covering of $X$.

1. $Y$ is closed in $X$ iff each $Y \cap X_i$ is closed in $X_i$.

2. If $X_i \subseteq Y$ for all $i$, then $Y = X$.
\end{lemma}

\begin{proof}
1. This is a superization of \cite{jan}, Part I, Lemma 1.13. The proof there given can be directly superized.

2. By superizing the fact (5) noted in \cite{jan}, Part I, Section 1.12, we see that the statement holds if 
$Y$ is closed in $X$. This last assumption is now satisfied by Part 1 above.
\end{proof}
\begin{prop}\label{ev of scheme}
Let $X$ be a superscheme. Then $X_{ev}$ is a closed subfunctor of $X$. In particular,
$X_{ev}$ is also a superscheme. 
\end{prop}
\begin{proof}
By Comparison Theorem $X$ can be identified with $Y^{\diamond}$, where $Y\in \mathcal{SV}$.  
By the definition, $y\in Y^{\diamond}(A)$ belongs to $(Y^{\diamond})_{ev}(A)$ iff $y$ factors through the canonical morphism 
$SSpec \ A\to SSpec \ A_0 , A\in\mathsf{SAlg}_K$. By Lemma \ref{canobijection}, $y\in (Y^{\diamond})_{ev}(A)$ iff for any open subset $V\subseteq Y^e$
we have $y^*(\mathcal{O}_Y(V))\subseteq \mathcal{O}_{SSpec \ A}((y^e)^{-1}(V))_0$. In particular, for any open subfunctor
$Z\subseteq Y$ we have $(Z^{\diamond})_{ev}=(Y^{\diamond})_{ev}\bigcap Z^{\diamond}$. Moreover, Proposition \ref{local}
implies that $(Y^{\diamond})_{ev}$ is a local $K$-functor. 
If $\{Y_i\}_{i\in I}$ is an open affine
covering of $Y$, then $\{Y^{\diamond}_i\}_{i\in I}$ is an open affine covering of $Y^{\diamond}$. Since $(Y^{\diamond})_{ev}\bigcap Y^{\diamond}_i$
is closed in $Y^{\diamond}_i$ for arbitrary $i\in I$, $(Y^{\diamond})_{ev}$ is closed in $Y^{\diamond}$ by Lemma \ref{Y=X}.
\end{proof}

\begin{prop}\label{ev of factor is factor of ev}
Let $G$ be an algebraic supergroup, and  let $H$ be a closed supersubgroup of $G$. Then we have 
$$(G \tilde{/} H)_{ev}=G_{ev} \tilde{/} H_{ev}, \quad (G \tilde{/} H)_{res}=G_{res} \tilde{/} H_{res}.$$
\end{prop}

\begin{proof}
Since by Corollary \ref{identifycor}, the second equality follows from the first, we concentrate on proving the first.

Theorem \ref{mainthm} combined with Proposition \ref{ev of scheme} implies that $(G \tilde{/} H)_{ev}$ is a superscheme.
Since it includes $(G_{ev} / H_{ev})_{(n)}$, we have $(G \tilde{/} H)_{ev} \supseteq G_{ev} \tilde{/} H_{ev}$. 

Recall from the arguments and the results in Sections \ref{proof2}, \ref{proof3}  
that $G \tilde{/} H$ and $G_{ev} \tilde{/} H_{ev}$ have (finite) open coverings $\{U_i\}$, $\{U'_i\}$, respectively,
such that $U_i~ \bigcap~ G_{ev} \tilde{/} H_{ev} = U'_i$ for each $i$. We see easily that $\{(U_i)_{ev}\}$ is an open covering
of $(G \tilde{/} H)_{ev}$. Suppose that we have proved 
\begin{equation}\label{eq-ev}
(U_i)_{ev} = U'_i\ \, \mbox{for each}\ \, i.
\end{equation}
Then the desired equality will follow from Lemma \ref{Y=X} on $X \supseteq Y$, $\{ X_i \}$ applied to  
$(G \tilde{/} H)_{ev} \supseteq G_{ev} \tilde{/} H_{ev}$, $\{ (U_i)_{ev} \}$. Therefore, we will aim to prove \eqref{eq-ev}, below.

By the same argument as given in the two paragraphs preceding Proposition \ref{prop-aim}, 
we may and we do suppose that we are in the situation of Section \ref{proof3}, 
and that the $U_i$ and the $U'_i$ above, with $i$ fixed, are given by $SSp~B^{coD}$, 
$SSp~\overline{B}^{co\overline{D}}$, respectively; here, $U_i$ may be supposed to be $SSp~B^{coD}$,
since $U_i$ is affine as was shown by the proof of Proposition \ref{prop-aim}.  Set $C = B^{coD}$. Since $(U_i)_{ev} = SSp~\overline{C}$ with
$\overline{C} = C/CC_1$, it follows that to prove the desired \eqref{eq-ev}, it suffices to prove 
\begin{equation}\label{overlineC}
\overline{C} = \overline{B}^{co\overline{D}}.
\end{equation}

Now, we know from Theorem \ref{superOberst} that $N \mapsto N \otimes_C B$ gives a category 
equivalence $\mathsf{SMod}_C \overset{\approx}{\longrightarrow} \mathsf{SMod}_B^D$, whose quasi-inverse is given by its
right adjoint $M \mapsto M^{coD}$; see \eqref{coinv}. Set
\begin{equation*}
\mathrm{gr}~B = \bigoplus_{i=0}^{\infty} B\mathfrak{a}^i/B\mathfrak{a}^{i+1}.
\end{equation*}
This is naturally a right $D$-comodule superalgebra. With respect to the natural $\mathbb{N}$-grading 
$(\mathrm{gr}~B)(i) = B\mathfrak{a}^i/B\mathfrak{a}^{i+1}$, $i \in \mathbb{N}$, this is an 
$\mathbb{N}$-graded object in $\mathsf{SMod}_B^D$. Note $\mathfrak{B} = (\mathrm{gr}~B)(0)$.
Since the equivalence $M \mapsto M^{coD}$ is exact,
we see that $(\mathrm{gr}~B)^{coD}$ coincides with
\begin{equation*}
\mathrm{gr}~C := \bigoplus_{i=0}^{\infty} (C \cap B\mathfrak{a}^i)/(C \cap B\mathfrak{a}^{i+1}).
\end{equation*}
Set $\mathfrak{C} = (\mathrm{gr}~C)(0)$; this coincides with $\mathfrak{B}^{coD}$. We see from Lemma \ref{injectivity} 
that the right $D$-comodule superalgebra $\mathfrak{B}$ satisfies 
the condition (2), and hence the remaining conditions, given in Part 1 of Theorem \ref{superOberst}. In particular,
\begin{equation*}
\beta_{\mathfrak{B}} : \mathfrak{B} \otimes_{\mathfrak{C}} \mathfrak{B} \to \mathfrak{B} \otimes D,\ 
\beta_{\mathfrak{B}}(b \otimes_{\mathfrak{C}} b') = b \rho' (b')
\end{equation*}
is a bijection, where $\rho' : \mathfrak{B} \to \mathfrak{B} \otimes D$ denotes the structure on $\mathfrak{B}$. 
 By applying $T/T^+ \otimes_T$ twice, it follows that the last bijection induces a bijection
$\overline{B} \otimes_{\mathfrak{C}} \overline{B} \overset{\simeq}{\longrightarrow} \overline{B} \otimes \overline{D}$.
We claim that $\overline{B}$ is faithfully flat over $\mathfrak{C}$, which together with the last bijection 
will imply that 
\begin{equation}\label{mathfrakC}
\mathfrak{C} = \overline{B}^{co\overline{D}}.
\end{equation}
To prove the claim, let $N \to N'$ be a morphism in $\mathsf{SMod}_{\mathfrak{C}}$. By the condition (4) for 
$\mathfrak{B}$, the morphism is an injection iff the induced morphism $N \otimes_{\mathfrak{C}} \mathfrak{B}
\to N' \otimes_{\mathfrak{C}} \mathfrak{B}$ in $\mathsf{SMod}_{\mathfrak{B}}^D$, and hence in $\mathsf{SMod}_T^D$, is
an injection. By the category equivalence \eqref{equiv}, the last condition is equivalent to that
$N \otimes_{\mathfrak{C}} \overline{B} \to N' \otimes_{\mathfrak{C}} \overline{B}$ is an injection. This proves
the claim. 

We see from \eqref{mathfrakC} that $\mathfrak{C} = (\mathrm{gr}~C)(0)$ is purely even. On the other hand, 
$(\mathrm{gr}~C)(1)$ is purely odd since so is $(\mathrm{gr}~B)(1)$. By construction, the product map $\bigotimes_{\mathfrak{B}}^i (\mathrm{gr}~B)(1) \to
(\mathrm{gr}~B)(i)$ is a surjection for each $i > 0$. 
We know already that $N \mapsto N \otimes_{\mathfrak{C}} \mathfrak{B}$ gives a category
equivalence $\mathsf{SMod}_{\mathfrak{C}} \approx \mathsf{SMod}_{\mathfrak{B}}^D$, which is indeed an 
equivalence of tensor categories, as is seen from Remark \ref{tensorfunctor}. 
It follows that 
the product map $\bigotimes_{\mathfrak{C}}^i (\mathrm{gr}~C)(1) \to
(\mathrm{gr}~C)(i)$ as well is a surjection for each $i$, or in other words, $\mathrm{gr}~C$ is generated by the elements of 
degree $\leq$ 1. Therefore, the $\mathbb{Z}_2$-grading on $\mathrm{gr}~C$ is such that
\begin{equation*}
(\mathrm{gr}~C)_0 = \bigoplus_{i\ \mathrm{even}} (\mathrm{gr}~C)(i),\quad (\mathrm{gr}~C)_1 = \bigoplus_{i\ \mathrm{odd}} (\mathrm{gr}~C)(i).
\end{equation*}
We can now apply \cite{amas}, Lemma 5.15 on $R$, $\mathcal{R}$ to the present $C$, $\mathrm{gr}~C$, to conclude
$CC_1^i = C \cap B\mathfrak{a}^i$ for all $i > 0$. The result for $i = 1$ shows
$\overline{C} = \mathfrak{C}$, which coincides with $\overline{B}^{co\overline{D}}$
by \eqref{mathfrakC}. This proves the desired \eqref{overlineC}.
\end{proof}

We define here \emph{affine morphisms}, and later \emph{faithfully flat morphisms};
see Definition \ref{def-faithfullmorph}.

\begin{definition}\label{def-affine}
Following \cite{dg}, we say that a morphism of $K$-functors ${\bf f} :
X\to Y$ is \emph{affine} if, for any $R\in\mathsf{SAlg}_K$ and any morphism
${\bf g} : SSp \ R\to Y$, the fiber product $X\times_{Y} SSp \ R$ is an
affine superscheme. 
\end{definition}
If $Y$ is an affine superscheme and ${\bf f} : X\to
Y$ is an affine morphism, then $X$ is obviously affine. In fact,
for ${\bf g}= \mathrm{id}_Y$ we have $X\simeq X\times_Y Y$.
\begin{lemma}\label{affine morphism}
If ${\bf f} : X\to Y$ is an affine morphism and $Y$ is a superscheme, then $X$ is
also a superscheme.
\end{lemma}
\begin{proof} Combining Yoneda's Lemma with superization of \cite{jan}, Part I, 1.7(3), one can copy the proof
of Proposition 3.3, \cite{dg}, I, \S 2 .
\end{proof}
\begin{lemma}\label{affinessislocal}
Let ${\bf f} : X\to Y$ be a morphism of local $K$-functors. If there is an open covering $\{Y_i\}_{i\in I}$ of $Y$ 
such that for each $i \in I$, the induced
morphism ${\bf f}^{-1}(Y_i)\to Y_i$ is affine, then ${\bf f}$ is affine.
\end{lemma}
\begin{proof}
Arguing as in Corollary 3.8, \cite{dg}, I, \S 2, one can reduce the general case to $Y=SSp \ R$. Moreover, one can assume that
$Y_i=D(Rg_i)$, where the even elements $g_i$ generate $R_0$ as an ideal. By our assumption, $\{{\bf f}^{-1}(Y_i)\}_{i\in I}$ is an open affine covering of $X$. Therefore, $X$ is a superscheme. Using Comparison Theorem one can translate our statement into the category
$\mathcal{SV}$. Denote ${\mathcal O}_X(|X|^e)$ by $A$. By Lemmas \ref{canobijection} and \ref{principalopen} we have a commutative diagram
$$\begin{array}{ccc}
|X| & \to & SSpec \ A \\
 \ \searrow & & \swarrow   \\
 & SSpec \ R &
\end{array}
$$
such that for any $i\in I$, in the induced (also commutative) diagram
$$\begin{array}{ccc}
|X|_{f^*(g_i)} & \to & U(Af^*(g_i)) \\
 \ \searrow & & \swarrow   \\
 & U(Rg_i) &
\end{array}
$$
the horizontal arrow is an isomorphism. Here $f=|{\bf f}|$. Thus $|X|\to SSpec \, A$ is also an isomorphism.
\end{proof}
Let ${\bf f} : X\to Y$ be a morphism of $K$-sheaves. We see that ${\bf f}$ is an epimorphism (of $K$-sheaves) iff the image of $X$ is dense (in the fppf topology) in $Y$,
or in other words, for any $A\in \mathsf{SAlg}_K, y\in Y(A)$, there is a fppf covering $B\geq A$ and $x\in X(B)$ such that ${\bf f}(B)(x)=Y(\iota^B_A)(y)$. For example, any quotient $G\to G \tilde{/} H$ is an epimorphism of $K$-sheaves.
Besides, if a superalgebra $B$ is a fppf covering of its supersubalgebra $A$, then $SSp \ \iota^B_A$ is an epimorphism of $K$-sheaves. 
\begin{lemma}\label{isomorphism} (see Example 2.6, \cite{dg}, III, \S 1) 
Assume that we have a commutative diagram of sheaves, 
$$\begin{array}{ccccl}
X  & \stackrel{{\bf f}}{\longrightarrow}  & Y \\
{\bf g} \ \searrow&&\swarrow \ {\bf h} & \\
& Z. &
\end{array}$$
If ${\bf p} : Z'\to Z$ is an epimorphism of sheaves and the induced morphism
$X\times_Z Z'\to Y\times_Z Z'$ is an isomorphism, then ${\bf f}$ is.
\end{lemma}
\begin{proof}
Suppose $x_1, x_2\in X(A)$, and ${\bf f}(A)(x_1)={\bf f}(A)(x_2)=y$. For an fppf covering $B\geq A$ there is $z\in Z'(B)$ such that 
$(y', z)$ belongs to $Y(B)\times_{Z(B)} Z'(B), y'=Y(\iota^B_A)(y)$. Moreover, $(x'_1, z), (x'_2, z)$ belong to $X(B)\times_{Z(B)} Z'(B),$ where
$x'_i=X(\iota^B_A)(x_i), i=1,2,$ and their images
in $Y(B)\times_{Z(B)} Z'(B)$ coincide. By the assumption, $x'_1=x'_2$ and since $X(\iota^B_A)$ is injective, $x_1=x_2$.
Analogously, for any $y\in Y(A)$, there are a fppf covering $B\geq A$ and $x'\in X(B)$ such that $X(B)(x')=
y'\in\mathrm{Ker}(Y(B)\begin{array}{c}\to \\ \to\end{array} Y(B\otimes_A B))$. Since ${\bf f}$ is an embedding, there is
a pre-image $x\in X(A)$ of $x'$ such that ${\bf f}(A)(x)=y$.
\end{proof}
\begin{prop}\label{cartesian} (cf. Corollary 2.12, \cite{dg}, III, \S 1)
Consider a cartesian square (or pullback square, see \cite{mc}, p.71)
$$\begin{array}{ccc}
X & \stackrel{{\bf f}}{\longrightarrow} & Y \\
{\bf p}\downarrow & & \downarrow {\bf q}\\ 
Z & \stackrel{{\bf g}}{\longrightarrow} & T
\end{array}
$$
of $K$-sheaves. Assume that ${\bf g}$ is an epimorphism and ${\bf p}$ is affine. Then ${\bf q}$ is affine.
\end{prop}
\begin{proof}
Let ${\bf g}_t : SSp \ B\to T, t\in T(B)$. There is $A\geq B$ and $x\in Z(A)$ such that ${\bf g}(A)(x)=T(\iota^A_B)(t)$. In particular,
${\bf g}{\bf g}_x= {\bf g}_t SSp \ \iota^A_B$. The square
$$\begin{array}{ccc}
SSp \ A\times_T Y & \stackrel{SSp \ \iota^A_B\times \mathrm{id}_Y}{\longrightarrow} & SSp \ B\times_T Y \\
\downarrow & & \downarrow \\ 
SSp \ A & \stackrel{SSp \ \iota^A_B}{\longrightarrow} & SSp \ B
\end{array}
$$   
is obviously cartesian (the vertical arrows are projections). Moreover, $SSp \ A\times_T Y\simeq SSp \ A\times_Z (Z\times_T Y)=SSp \ A\times_Z X$ is an affine superscheme. In other words, one can assume that $X=SSp \ D, Z=SSp \ A, T=SSp \ B$. Then $X\times_Y X\simeq SSp \ A\times_{SSp \ B} SSp \ D$ is affine and the projections ${\bf pr}_i : X\times_Y X\to X, i=1, 2,$ are dual to superalgebra morphisms $i_s : D\to K[X\times_Y X], s=1,2$.
Set $N=\ker(i_1-i_2)$. $N$ is a supersubalgebra of $D$ and $(SSp \ N, SSp \ \iota^D_N)$ is a cokernel of $({\bf pr}_1, {\bf pr}_2)$. Define a subfunctor $Y'\subseteq Y$ by $Y'(A)={\bf f}(A)(X(A)), A\in \mathsf{SAlg}_K$.
There is $Y'\to SSp \ N$ given by $y\mapsto (SSp \ \iota^D_N)(A)(x)$ whenever $y={\bf f}(A)(x), x\in X(A)$. Since $Y'$ is dense in $Y$, there is ${\bf s} : Y\to SSp \ N$ such that $SSp \ \iota^D_N ={\bf s}{\bf f}$. Observe that ${\bf h}{\bf pr}_1={\bf h}{\bf pr}_2$, where ${\bf h}={\bf q}{\bf f}={\bf g}{\bf p}$. Therefore, there is an unique morphism ${\bf m} : SSp \ N\to SSp \ B$ such that $(SSp \ \iota^D_N){\bf m}={\bf h}$. 
It is sufficient to check that $SSp \ A\times_{SSp B} \ {\bf s}$ is an isomorphism and apply Lemma \ref{isomorphism}. The functor $SSp \ A\times_{SSp \ B} \ ?$ takes the above cartesian square to a cartesian square
$$\begin{array}{ccc}
SSp \ A\otimes_B D & \stackrel{{\bf f}'}{\longrightarrow} & SSp \ D \\
{\bf p}' \downarrow & & \downarrow {\bf q}' \\ 
SSp \ A\otimes_B A & \stackrel{{\bf g}'}{\longrightarrow} & SSp \ A
\end{array} 
$$   
with ${\bf n}'= SSp \ A\times_{SSp \ BB} \ {\bf n}, {\bf n}\in\{{\bf f}, {\bf p}, {\bf q}, {\bf g}\}$. Finally, the corresponding cokernel
is isomorphic to $(SSp \ A\otimes_B N, {\bf s}')$ and our statement is obvious. 
\end{proof}
\begin{definition}\label{def-faithfullmorph}
A morphism of superschemes ${\bf f} : X\to Y$ is called \emph{flat}, if for any open affine supersubschemes $SSp \ A\subseteq X$ and
$SSp \ B\subseteq Y$ such that ${\bf f}(SSp \ A)\subseteq SSp \ B$, the superalgebra $A$ is a flat $B$-module.
The morphism $\bf f$ is called \emph{faithfully flat}, if it is flat and $|{\bf f}|^e$ is surjective.
\end{definition}
Remind that if a $B$-superalgebra $A$ is a faithfully flat $B$-module, then $(SSpec \ \iota^A_B)^e$ is surjective. 
In particular, if $X$ and $Y$ have open affine coverings indexed by the same set, say $\{X_i\}_{i\in I}$ and $\{Y_i\}_{i\in I}$, respectively, 
such that for every $i\in I$, ${\bf f}(X_i)\subseteq Y_i$, and $K[X_i]$ is a faithfully flat $K[Y_i]$-module, then ${\bf f}$ is faithfully flat by the Comparison Theorem.
\begin{corollary}\label{affandfaith}
Assume that an algebraic supergroup $G$ acts freely on an affine superscheme $X$. If the quotient $X \tilde{/} G$ is
a superscheme, then the quotient morphism $X \to X \tilde{/} G$
is affine and faithfully flat.
\end{corollary}
\begin{proof}
Denote the quotient morphism by $\pi$. By Lemma \ref{affinessislocal} one has to show that for an open affine supersubscheme $SSp \ R=V\subseteq X \tilde{/} G$ 
the open subfunctor $U=\pi^{-1}(V)$ is also affine, say $U\simeq SSp \ A$, and $R\leq A$. Using Proposition \ref{cartesian} one can easily superize \cite{jan}, Part I, 5.7(1).
\end{proof}
\begin{rem}\label{diffway}
If $\mathrm{char} \, K > 0$, then the statement that $G\tilde{/} H$ is a superscheme can be proved in a different way.
The following observation plays crucial role in this proof. If $H_1\leq H_2$ are closed
supersubgroups of $G$, $G \tilde{/} H_2$ is a superscheme and $H_2 \tilde{/} H_1$
is an affine superscheme, then $G \tilde{/} H_1$ is a superscheme.
In fact, the following diagram (cf. \cite{cps}, p. 8)
$$\begin{array}{ccc}
G\times H_2 \tilde{/} H_1 & \to & G \tilde{/} H_1 \\
\downarrow & & \downarrow \\
G & \to & G \tilde{/} H_2
\end{array},$$
where $G\times H_2 \tilde{/} H_1\to G$ is the canonical projection and
$G\times H_2\tilde{/} H_1 \to  G\tilde{/}H_1$ is induced by $(g, hH_1(A))\mapsto
ghH_1(A), g\in G(A), h\in H_2(A), A\in \mathsf{SAlg}_K$, is a cartesian
square. By Proposition \ref{cartesian} the canonical morphism $G\tilde{/}H_1\to G\tilde{/}H_2$
is affine. It remains to apply Lemma \ref{affine morphism}. 

Now, denote by $G_1$ the
first infinitesimal supersubgroup of $G$ (cf. \cite{z}, p.738). The supergroup
$G_1$ is finite and normal in $G$. Moreover, the supergroup $G \tilde{/} G_1$ is purely even.
By Lemma \ref{G_ev/H_ev} $$G \tilde{/} HG_1\simeq (G \tilde{/} G_1) \tilde{/} (HG_1 \tilde{/} G_1)$$ 
is a superscheme and $HG_1\tilde{/} H\simeq G_1\tilde{/} G_1 \bigcap H$
is affine by Theorem 7.1 from \cite{z} (or by Corollary \ref{cor1}). The above observation completes the proof.
\end{rem}
\begin{rem}\label{folklore}
The following result seems a folklore: if $G$ is an algebraic group over an arbitrary field, and if $H$ is its closed, geometrically reductive subgroup,
then $G \tilde{/} H$ is affine.
\end{rem} 
\begin{proof}
By \cite{wh}, $H$ is geometrically reductive iff $H^0_{red}$ is reductive.
If $\mathrm{char} \, K =0$, then any algebraic group is reduced, and $G\tilde{/} H\simeq (G\tilde{/} H^0 )\tilde{/}(H\tilde{/} H^0)$ is affine by
\cite{cps}, Corollary 4.5, and by \cite{jan}, Part I, 5.5(6). If $\mathrm{char} \, K > 0$, then by arguing as above, one can assume that $H$ is connected. There is an integer $n> 0$ such that the induced morphisms
$G_{red}\to G\tilde{/} G_n$ and $H_{red}\to H\tilde{/} H_n$ are epimorphisms; cf. \cite{wh}. Since $HG_n\tilde{/} G_n \simeq H\tilde{/} H_n$ is a reduced, geometrically reductive 
group (cf. \cite{bs}, p.70), it follows again by \cite{cps}, Corollary 4.5 that $G\tilde{/} HG_n \simeq
(G\tilde{/} G_n)\tilde{/} (HG_n\tilde{/} G_n)$ is affine. It remains to observe that $H G_n \tilde{/} H \simeq
G_n \tilde{/} H_n$ is affine, and to argue as in Remark \ref{diffway}. 
\end{proof}

\section{Bosonization technique}\label{bosonization}

In this section that is rather independent of the preceding ones, we choose arbitrarily a non-zero
commutative ring $K$, and let $K$ be the ground ring over which we work. 
We aim to give the postponed proof of (2) $\Leftrightarrow$ (3) of Theorem \ref{superOberst}, Part 1,
in a generalized situation. We do not assume that (Hopf) algebras are commutative.
 
Let $J$ be a Hopf algebra, and assume that the antipode of $J$, which we denote by $S : J \to J$,
is bijective. Let $\Delta : J \to J \otimes J$,\ $\varepsilon : J \to K$ denote the coproduct and the counit of $J$,
respectively. For this coproduct we use the sigma notation
\begin{equation*}
\Delta(x) = \sum x_1 \otimes x_2, \quad (\Delta \otimes \mathrm{id}) \circ \Delta (x) = (\mathrm{id} \otimes \Delta) \circ \Delta (x) =
\sum x_1 \otimes x_2 \otimes x_3.
\end{equation*}
We have the braided tensor category ${}^J_J \mathcal{YD}$ of \emph{Yetter-Drinfeld modules}
with left $J$-action and left $J$-coaction. To be more precise, an object, say $V$, in ${}^J_J \mathcal{YD}$ 
is a left $J$-module and left $J$-comodule, whose structures we denote by
\begin{equation*}
J \otimes V \to V,\ x \otimes v \mapsto \ x \rightharpoonup v; \quad  
\lambda : V \to J \otimes V,\ \lambda (v) = \sum v_J \otimes v_V;
\end{equation*}
these are required to satisfy 
\begin{equation*}
\lambda (x \rightharpoonup v) = \sum x_1v_{J}S(x_3) \otimes (x_2 \rightharpoonup v_{V}),\quad x \in J,\ v \in V.
\end{equation*}
The category ${}^J_J \mathcal{YD}$ has the same tensor product just as 
the category of left $J$-(co)modules, and has the braiding given by
\begin{equation*}
V \otimes W \overset{\simeq}{\longrightarrow} W \otimes V, \quad     
v \otimes w \mapsto \sum (v_J \rightharpoonup w) \otimes v_V,
\end{equation*}
where $V, W \in {}^J_J \mathcal{YD}$; see \cite{mon}, Section 10.6. 

Suppose that $J$ is the group algebra $K\mathbb{Z}_2$ of the group $\mathbb{Z}_2$. 
Regard each object $V \in \mathsf{SMod}_K$ as a left $J$-module by letting the generator of 
$\mathbb{Z}_2$ act on homogeneous elements $v \in V$ by \eqref{Z2action}.
Then, $V$ turns into an object in 
${}^{J}_{J}\mathcal{YD}$. We can thus embed $\mathsf{SMod}_K$ into 
${}^{J}_{J}\mathcal{YD}$ as a braided tensor full subcategory. Therefore,
results on ${}^J_J\mathcal{YD}$ can apply to $\mathsf{SMod}_K$.

Suppose again that $J$ is general. Let $D$ be a Hopf algebra object in ${}^J_J \mathcal{YD}$.
Radford's \emph{biproduct construction} \cite{rad} (or in recent terms, \emph{bosonization}) constructs an ordinary 
Hopf algebra,
\begin{equation*}
\Hat{D} := D \cmddotrtimes J.
\end{equation*}
As an algebra, this is a smash product; see \cite{sw}, Section 7.2.

\begin{lemma}\label{antipode}
The antipode of $D$ is bijective iff the antipode of $\Hat{D}$ is bijective.
\end{lemma}

\begin{proof}
This is seen from \cite{rad}, Proposition 2.
\end{proof}

Regard $D$ as an ordinary coalgebra, and let $\mathsf{Mod}^D$ denote the category of right $D$-comodules, as before.
Since $J$ is a Hopf subalgebra of $\Hat{D}$, we can define the category $\mathsf{Mod}^{\Hat{D}}_J$ 
of right $(\Hat{D}, J)$-Hopf modules; see \cite{tak}.  Note that $J$ is an algebra object in the tensor category 
$\mathsf{Mod}^{\Hat{D}}$ of right $\Hat{D}$-comodules, and $\mathsf{Mod}^{\Hat{D}}_J$ is precisely
the category $(\mathsf{Mod}^{\Hat{D}})_J$ of the right $J$-module object in $\mathsf{Mod}^{\Hat{D}}$.

Suppose $N \in \mathsf{Mod}^D$. Let $\rho : N \to N \otimes D$, $\rho (n) = \sum n_0 \otimes n_1$ 
denote the structure. We set $\Hat{N} = N \otimes J$, and define 
$\Hat{\rho} : \Hat{N} \to \Hat{N} \otimes \Hat{D}$ by
\begin{equation*}
\Hat{\rho} (n \otimes x) = \sum (n_0 \otimes (n_1)_J x_1) \otimes ((n_1)_D \otimes x_2),\quad n \otimes x \in N \otimes J = \Hat{N}.
\end{equation*}
In particular, $D$ is in $\mathsf{Mod}^D$ with respect to the coproduct $\Delta : D \to D \otimes D$. For this, 
$\Hat{\Delta} :  \Hat{D} \to \Hat{D} \otimes \Hat{D}$ coincides with the coproduct of $\Hat{D}$.
Let $J$ act on $\Hat{N} = N \otimes J$ by the right multiplication on the factor $J$.

\begin{prop}\label{equivalence}
$\Hat{N} = (\Hat{N}, \Hat{\rho})$ is an object in $\mathsf{Mod}^{\Hat{D}}_J$, and 
\begin{equation}\label{Hat}
N \mapsto \Hat{N},\ \mathsf{Mod}^D \to \mathsf{Mod}^{\Hat{D}}_J
\end{equation}
gives a functor. Moreover, this is a category equivalence.
\end{prop}

\begin{proof}
The first sentence is directly verified. To prove the category equivalence, let $M \in \mathsf{Mod}^{\Hat{D}}_J$. 
Define $J^+ = \mathrm{Ker}~\varepsilon$.
One sees that the structure morphism $M \to M \otimes \Hat{D}$, composed with $\mathrm{id} \otimes \mathrm{id} \otimes \varepsilon : 
M \otimes D \otimes J \to M \otimes D$, induces a right $D$-comodule structure on $M/MJ^+$, and 
\begin{equation}\label{bar}
M \mapsto M/MJ^+,\  \mathsf{Mod}^{\Hat{D}}_J \to \mathsf{Mod}^D
\end{equation}
gives a functor. We wish to prove that the functors \eqref{Hat} and \eqref{bar} are quasi-inverses of each other.
We see that $\mathrm{id} \otimes \varepsilon : N \otimes J \to N$ induces an isomorphism 
\begin{equation*}
\Hat{N}/\Hat{N}J^+ \overset{\simeq}{\longrightarrow} N 
\end{equation*}
in $\mathsf{Mod}^D$ which is natural in $N$. On the other hand, the structure morphism $M \to M \otimes \Hat{D}$, composed
with $\mathrm{id} \otimes \varepsilon \otimes \mathrm{id} : M \otimes D \otimes J \to M \otimes J$, induces a morphism
\begin{equation*}
M \to M/MJ^+ \otimes J = \widehat{M/MJ^+}
\end{equation*}
in $\mathsf{Mod}^{\Hat{D}}_J$ which is natural in $M$. This last is indeed an isomorphism since it coincides with 
the well-known Hopf-module isomorphism (see \cite{sw}, Theorem 4.1.1), if we regard $M$ naturally as an object in $\mathsf{Mod}^J_J$. 
The last two natural
isomorphisms prove the desired equivalence.
\end{proof}

We keep the notation as above.
Let $B$ be a non-zero left $J$-module algebra, that is, an algebra object in the tensor category ${}_J \mathsf{Mod}$ of 
left $J$-modules. Suppose in addition that $B$ is a right $D$-comodule whose structure morphism, denoted as before by
$\rho : B \to B \otimes D$, $\rho(b) = \sum b_0 \otimes b_1$, is left $J$-linear and satisfies the following braid relation:
\begin{equation*}
\rho (bb') = \sum b_0 ((b_1)_J \rightharpoonup b'_0) \otimes (b_1)_D b'_1, \quad b,~ b' \in B.
\end{equation*}
The assumptions are satisfied if $B$ is an algebra object in $({}^J_J\mathcal{YD})^D$. Set 
$C = B^{coD}$; see \eqref{coinv}.
This $C$ is a left $J$-module subalgebra of $B$. 
Just as in Theorem \ref{superOberst}, we have the maps  
\begin{equation*} 
\alpha = \alpha_B : B \otimes B \to B \otimes D, \quad \alpha(b \otimes b') = b \rho (b'),
\end{equation*}
\begin{equation*}
\beta = \beta_B : B \otimes_C B \to B \otimes D, \quad \beta(b \otimes_C b') = b \rho (b').
\end{equation*}
Obviously, $\alpha_B$ is a surjection iff $\beta_B$ is.
Regard $\Hat{B} = B \rtimes J$ as the smash-product algebra.  Identify $C$ with 
$C \otimes K$ in $\Hat{B}$.

\begin{prop}\label{BHatB} Let the notation be as above.

1. $\Hat{B} = (\Hat{B}, \Hat{\rho})$ is a right $\Hat{D}$-comodule algebra such that 
$C = \Hat{B}^{co\Hat{D}}$.

2. The map $\beta_B$ above is a surjection (resp., bijection) iff 
the map
\begin{equation*}
\beta_{\Hat{B}} : \Hat{B} \otimes_C \Hat{B} \to \Hat{B} \otimes \Hat{D}, \quad 
\beta_{\Hat{B}}(b \otimes_C b') = b \Hat{\rho} (b')
\end{equation*}
is a surjection (resp., bijection).
\end{prop}

\begin{proof}
1. One sees directly that $\Hat{B}$ is a right $\Hat{D}$-comodule algebra. Obviously, 
$C \subseteq \Hat{B}^{co\Hat{D}}$. To prove the converse inclusion, let 
$\mathbf{b} = \sum_i b_i \otimes x_i \in \Hat{B}^{co\Hat{D}}$. Then we have 
in $B \otimes J \otimes D \otimes J$,
\begin{equation}\label{sigma}
\sum_i (b_i)_0 \otimes (b_i)_{1J} (x_i)_1 \otimes (b_i)_{1D} \otimes (x_i)_2 
= \sum_i b_i \otimes x_i \otimes 1 \otimes 1.
\end{equation}
Applying $\mathrm{id} \otimes \varepsilon \otimes \varepsilon \otimes \mathrm{id}$, we have
$\mathbf{b} = \sum_i b_i \varepsilon (x_i) \otimes 1$. Hence we may suppose $\mathbf{b} = b \otimes 1$ 
with $b \in B$, in which case we see $b \in C$, as desired, by applying 
$\mathrm{id} \otimes \varepsilon \otimes \mathrm{id} \otimes \varepsilon$ to the both sides of \eqref{sigma}.

2. Let $S^- : J \to J$ denote the composite-inverse of the antipode $S$ of $J$. 
We define maps,
\begin{equation*}
\phi : B \otimes J \otimes D \to B \otimes D \otimes J, \quad 
\phi(b \otimes x \otimes d) = \sum b \otimes d_D \otimes xS(d_J),
\end{equation*}
\begin{equation*}
\psi : B \otimes D \otimes J \to B \otimes D \otimes J, \quad 
\psi (b \otimes d \otimes x) = \sum (S^- (x_1) \rightharpoonup b) \otimes d \otimes x_2.
\end{equation*}
These are bijections, whose inverses are given by
\begin{equation*}
\phi^{-1} (b \otimes d \otimes x) = \sum b \otimes xd_J \otimes d_D, 
\end{equation*}
\begin{equation*}
\psi^{-1} (b \otimes d \otimes x) = \sum (x_1 \rightharpoonup b) \otimes d \otimes x_2.
\end{equation*} 
We have a well-defined map, 
\begin{equation*}
\theta : (B \otimes_C B) \otimes J \to \Hat{B} \otimes_C B,\quad \theta ((b \otimes_C b') \otimes x) 
= \sum ((x_1 \rightharpoonup b) \otimes x_2) \otimes_C b',
\end{equation*}
which is a bijection with inverse
\begin{equation*}
\theta^{-1}((b \otimes x) \otimes_C b') = \sum ((S^-(x_1) \rightharpoonup b) \otimes_C b') \otimes x_2.
\end{equation*}
One sees that the map $\beta_{\Hat{B}}$ for $\Hat{B}$ is a surjection (resp., bijection) 
iff 
\begin{equation*}
\gamma : \Hat{B} \otimes_C B \to \Hat{B} \otimes D, \quad 
\gamma ((b \otimes x) \otimes_C b') = \sum (b(x_1 \rightharpoonup b'_0) \otimes x_2 (b'_1)_J) \otimes (b'_1)_D
\end{equation*}
is, since the former is the composite of the base extension $\gamma \otimes \mathrm{id}_J$ (of $\gamma$ 
along $K \to J$) with the bijection
\begin{equation*}
\Hat{B} \otimes \Hat{D} \overset{\simeq}{\longrightarrow} \Hat{B} \otimes \Hat{D}, \quad 
(b \otimes x) \otimes (d \otimes y) \mapsto \sum (b \otimes xy_1) \otimes (d \otimes y_2).
\end{equation*}
The desired result follows since we see that the base extension $\beta_{B} \otimes \mathrm{id}_J$ 
coincides with the composite 
$\psi \circ \phi \circ \gamma \circ \theta$.
\end{proof} 

\begin{prop}\label{braidedSchneider}
Assume that $K$ is a field, and the antipode of $D$ is bijective. Then the following are 
equivalent:
\begin{itemize}
\item[(1)]
\begin{itemize}
\item[(a)]
$B$ is injective as a right $D$-comodule, and
\item[(b)]
$\alpha_B : B \otimes B \to B \otimes D$ is a surjection;
\end{itemize}
\item[(2)]
\begin{itemize}
\item[(a)]
$B$ is faithfully flat as a left $C$-module, and
\item[(b)]
$\beta_B : B \otimes_C B \to B \otimes D$ is a bijection;
\end{itemize}
\item[(3)]
\begin{itemize}
\item[(a)]
$B$ is faithfully flat as a right $C$-module, and
\item[(b)]
$\beta_B : B \otimes_C B \to B \otimes D$ is a bijection.
\end{itemize}
\end{itemize}
\end{prop}

\begin{proof}
By Lemma \ref{antipode}, the antipode of $\Hat{D}$ is bijective. Combining this with Part 1 of Proposition 
\ref{BHatB}, we can apply Schneider's Theorem \cite{sch}, Theorem I,
to our $\Hat{B},~ C,~\Hat{D}$, in order to see that the following three conditions are equivalent.
\begin{itemize}
\item[($\Hat{1}$)]
\begin{itemize}
\item[(a)]
$\Hat{B}$ is injective as a right $\Hat{D}$-comodule, and
\item[(b)]
$\beta_{\Hat{B}} : \Hat{B} \otimes_C \Hat{B} \to \Hat{B} \otimes \Hat{D}$ is a surjection;
\end{itemize}
\item[($\Hat{2}$)]
\begin{itemize}
\item[(a)]
$\Hat{B}$ is faithfully flat as a left $C$-module, and
\item[(b)]
$\beta_{\Hat{B}} : \Hat{B} \otimes_C \Hat{B} \to \Hat{B} \otimes \Hat{D}$ is a bijection;
\end{itemize}
\item[($\Hat{3}$)]
\begin{itemize}
\item[(a)]
$B$ is faithfully flat as a right $C$-module, and
\item[(b)]
$\beta_{\Hat{B}} : \Hat{B} \otimes_C \Hat{B} \to \Hat{B} \otimes \Hat{D}$ is a bijection.
\end{itemize}
\end{itemize}
By Proposition \ref{BHatB}, Part 2, we see (2) $\Leftrightarrow$ ($\Hat{2}$), (3) $\Leftrightarrow$ ($\Hat{3}$). 
To see the latter, note that 
\begin{equation*}
J \otimes B \to \Hat{B},\quad x \otimes b \mapsto \sum (x_1 \rightharpoonup b) \otimes x_2
\end{equation*}
gives a right $B$-linear isomorphism, where on the left-hand side, we let $B$ act by the right multiplication 
on the factor $B$. We thus have (2) $\Leftrightarrow$ (3). We see (3) $\Rightarrow$ (1), since (3) implies that
$B$ is faithfully coflat as a right $D$-comodule,
which is equivalent to (1)(a). To complete the proof we have only to prove (1) $\Rightarrow$ ($\Hat{1}$).

Assume (1). By Proposition \ref{BHatB}, Part 2, we have ($\Hat{1}$)(b).
By Proposition \ref{equivalence}, $\Hat{B}$ is an injective object in $\mathsf{Mod}^{\Hat{D}}_J$, 
whence the structure morphism $\Hat{B} \to \Hat{B} \otimes \Hat{D}$ splits in $\mathsf{Mod}^{\Hat{D}}_J$, 
implying ($\Hat{1}$)(a); see \cite{tak1}, Proposition A.2.1. This completes the proof. 
\end{proof}

Now we can give the postponed proof.

\vspace{\baselineskip}
\noindent
\emph{Proof of (2) $\Leftrightarrow$ (3) in Theorem \ref{superOberst}}. 
When $J = K \mathbb{Z}_2$, in particular, the proposition above can apply to the super context and when
$D$ is a supercommutative Hopf superalgebra, whose antipode is necessarily involutory and hence is bijective.
Then the desired equivalence results.
$\square$

\section*{Acknowledgments}
The first named author was supported by
Grant-in-Aid for Scientific Research (C) 20540036 and by Grant-in-Aid for Foreign
JSPS Fellow 21\,$\cdot$\,09219 both from Japan Society of the Promotion of Science.
The second named author was supported by FAPESP 09/50981-2 and by RFFI 10.01.00383 a.
He also thanks Professor Ivan Shestakov whose efforts played crucial role to manage his visit to Sao Paulo University.  
The authors both thank the referees for their kind suggestions, which improved the exposition 
of this paper.

\end{document}